\documentclass[11pt]{amsart}

\usepackage{geometry}
\usepackage{amsmath, amsthm}
\usepackage{amssymb}
\usepackage{tikz-cd}
\usepackage{enumerate}
\usepackage{hyperref}
\hypersetup{
	colorlinks = true,
	linkcolor = {blue},
	citecolor={magenta}
}

\theoremstyle{plain}
\newtheorem{theorem}{Theorem}[section]
\newtheorem{prop}[theorem]{Proposition}
\newtheorem{corr}[theorem]{Corollary}
\newtheorem{lemma}[theorem]{Lemma}

\theoremstyle{definition}
\newtheorem{defn}[theorem]{Definition}
\newtheorem{example}[theorem]{Example}
\newtheorem{obs}[theorem]{Observation}
\newtheorem{question}[theorem]{Question}
\newtheorem{remark}[theorem]{Remark}

\foreach \x in {A,...,Z}{\expandafter\xdef\csname cal\x\endcsname{\noexpand\mathcal{\x}}}
\foreach \x in {A,...,Z}{\expandafter\xdef\csname bb\x\endcsname{\noexpand\mathbb{\x}}}
\foreach \x in {A,...,Z}{\expandafter\xdef\csname fr\x\endcsname{\noexpand\mathfrak{\x}}}
\foreach \x in {a,...,z}{\expandafter\xdef\csname fr\x\endcsname{\noexpand\mathfrak{\x}}}
\foreach \x in {A,...,Z}{\expandafter\xdef\csname rm\x\endcsname{\noexpand\mathrm{\x}}}
\foreach \x in {a,...,z}{\expandafter\xdef\csname rm\x\endcsname{\noexpand\mathrm{\x}}}

\DeclareMathOperator{\rk}{rk}
\DeclareMathOperator{\cork}{corank}
\DeclareMathOperator{\im}{Im}

\DeclareMathOperator{\Op}{\mathcal{O}p}
\DeclareMathOperator{\Sym}{Sym}
\DeclareMathOperator{\Sol}{Sol}
\newcommand{\Span}{\textrm{Span}}

\begin{document}
\title{On Horizontal Immersions of Discs in Fat Distributions of Type $(4,6)$}
\author[A. Bhowmick]{Aritra Bhowmick}

\address{Statistics and Mathematics Unit, Indian Statistical Institute\\ 203, B.T. Road, Kolkata 700108, India}
\email{avowmix@gmail.com}

\subjclass[2010]{58A30, 58J05, 58A15, 35J60, 53C23}

\keywords{h-principle, fat distribution, holomorphic contact distribution, elliptic PDE}

\begin{abstract}
In this article we discuss horizontal immersions of discs in certain corank-$2$ fat distributions on $6$-dimensional manifolds. The underlying real distribution of a holomorphic contact distribution on a complex $3$ manifold belongs to this class. The main result presented here says that the associated nonlinear PDE is locally invertible. Using this we prove the existence of germs of embedded horizontal discs.
\end{abstract}

\maketitle

\section{Introduction}
In subriemannian geometry one studies distributions on smooth manifolds. Bracket-generating distributions, which lie at the opposite end of the integrable ones, are the primary focus. A distribution $\calD\subset TM$ is \emph{bracket generating} if successive Lie brackets of local vector fields in $\calD$ around any point $x\in M$ span the tangent space $T_xM$.

\subsection{Immersions In a Manifold with a Distribution}
Given a distribution $\calD$, we can consider smooth curves $\gamma: \bbI \to M$ which are everywhere tangent to it. We shall call them \emph{$\calD$-horizontal curves} or simply \emph{horizontal curves}. These curves play an important role in understanding the distribution. In fact, if $\calD$ is bracket generating, any two points of the manifold $M$ can be joined by a smooth horizontal curve (\cite{chowBracketGenerating}) and the space of horizontal curves joining two points has the same homotopy type as the space of smooth curves joining them (\cite{geHorizontalCC,gromovCCMetric}).

More generally, one may consider horizontal immersions of $k$-manifolds in a manifold $M$ endowed with a distribution $\calD$. Such maps can be thought of as the solutions to the differential operator,
\begin{align*}
	\frD : C^\infty(\Sigma, M) &\to \Omega^1(\Sigma,\bbR^p) = \Gamma \hom(T\Sigma, \bbR^p)\\
	f &\mapsto \big(f^*\lambda^1,\ldots, f^*\lambda^p\big)
\end{align*}
where we assume that $\calD$ is a corank-$p$ distribution on $M$, given as the common kernel, $\calD = \cap_{i=1}^p \ker\lambda^i$, for $\lambda^i \in\Omega^1(M), i=1,\ldots, p$. Gromov defines (\cite[pg. 338]{gromovBook}) a $\calD$-horizontal immersion $f:\Sigma\to M$ to be \emph{regular} if the \emph{algebraic} system
$$d\lambda^i \big(\partial, \; u_* X_j\big) = \sigma_{ij}, \quad 1\le i \le p,\; 1\le j \le \dim\Sigma$$
is solvable for (local) vector fields $\partial \in \Gamma f^*\calD$, for any given set of arbitrary smooth functions $\{\sigma_{ij}\}$. Here, $\{X_j\}$ is some fixed local framing of $T\Sigma$. It follows that the linearization of $\frD$ at some $f : \Sigma\to M$ is invertible provided $f$ is regular. In fact, the linearization operator at a regular map $f$ has a $0^\text{th}$-order inversion. Then, using a version of the Nash implicit function theorem, Gromov (\cite{gromovBook}) proceeds to obtain the $h$-principle for regular $\calD$-horizontal immersions. This problem has been revisited with further details in \cite{gromovCCMetric, pansuCarnotManifold}.

Now, for any horizontal immersion $f$, $\im df_x$ is an isotropic subspace of $\calD_{f(x)}$. So, accounting for this isotropy, the above algebraic system is \emph{underdetermined} whenever $$\rk\calD - \dim\Sigma > \dim\Sigma \times \cork\calD.$$ Gromov proves (\cite[pg. 256]{gromovCCMetric}) that a \emph{generic} distribution $\calD$ on $M$ admits horizontal germs of immersions through generic points of $M$, under the above inequality. Then he proceeds to state the $h$-principle (\cite[pg. 258]{gromovCCMetric}) for $\calD$-horizontal immersions of $\Sigma \to M$, with some additional regularity condition, whenever $$\rk \calD\geq (\dim\Sigma + 1)(\cork\calD+1)$$
holds. He also conjectures (\cite[pg. 259]{gromovCCMetric}) that one may be able to improve the inequality to $\dim M\geq (\dim\Sigma +1)(\cork\calD+1)$.

\subsection{Horizontal Immersions in a Manifold with a Fat Distribution}
Fat distributions (see \autoref{sec:fatDistrbution}) form an interesting class of bracket generating distributions; a distribution $\calD$ is called \emph{fat} (or \emph{strongly bracket generating}) if any non-vanishing local section of $\calD$ Lie bracket generates the tangent space in $1$-step. Contact distributions are the prime examples of fat distributions on an odd dimensional manifold $M$; they are locally defined by 1-forms $\alpha$ such that $d\alpha$ is non-degenerate on $\ker\alpha$. Horizontal immersions for such distributions are well understood; in fact, they satisfy the parametric $h$-principle (\cite{duchampLegendre,eliashbergBook}).

In corank $1$, fatness is a generic property. However, for higher corank, fat distribution germs are never generic (see \autoref{rmk:fatNonGeneric}). In corank $2$, the most prominent examples of fat distributions are given by the holomorphic analogue of contact structures (see \autoref{exmp:contactHolomorphic}). These manifolds are modeled on the holomorphic 1-jet space $J^1(\bbC^n,\bbC)$ like their real counterparts, and 1-jet prolongation of any holomorphic  map $\bbC^n\to \bbC$ is a holomorphic Legendrian embedding. So, there are plenty of holomorphic horizontal submanifolds in any holomorphic contact manifold. In \cite{forstnericLegendrianCurves} the authors have shown that holomorphic Legendrian embeddings of an open Riemann surface $\Sigma$ into the standard holomorphic contact manifold $\big(\bbC^{2n+1},dz-\sum_i y_i dx_i\big)$ satisfy the parametric Oka principle. In particular, they prove that the space of Legendrian holomorphic embeddings $\Sigma\hookrightarrow \bbC^{2n+1}$ has the same homotopy type as the space of continuous maps $\Sigma\to \bbS^{4n-1}$. The authors further observe that such a global $h$-principle type result may not be true for a general holomorphic contact manifold.

Now, recall that for a given (real) contact structure $\ker\alpha$, we have the \emph{Reeb} vector field $R$, which is defined by $$\iota_R d\alpha|_{\ker\alpha} = 0,\quad \alpha(R) = 1.$$
Similarly, given a holomorphic contact structure $\Xi$ on a complex manifold $M$, the underlying real distribution $\calD$ can be locally written as $\calD = \ker\lambda^1 \cap \ker\lambda^2$ and furthermore, we can identify (local) vector fields $Z_1, Z_2$, such that $TM = \calD \oplus \Span \langle Z_1, Z_2\rangle$ and $Z_1, Z_2$ display Reeb-like properties. This motivates the definition of \emph{local Reeb like directions} (\autoref{defn:fatWithReeb}). Holomorphic contact distributions are the best known examples in the class of fat corank $2$ distributions considered in this article, which admit local Reeb directions.

Gromov's general theorem for regular horizontal immersions $\Sigma \to M$ is not applicable to such a distribution $\calD$ on a $6$-dimensional manifold $M$, whenever $\dim\Sigma = 2$, as there cannot exist any `regular' map due to dimension constraints. And yet the results of \cite{forstnericLegendrianCurves} indicates that there is still a possibility of obtaining an $h$-principle.

\subsection{Main Results}
In this article, we consider $\calD$ to be a corank $2$ fat distribution on a manifold of dimension $6$, which admits (local) Reeb directions, and then we study horizontal immersions of the closed unit disc $\Sigma=\mathbb D^2$ into $M$. 
The main theorem of this article may be stated as follows.
\begin{theorem}\label{thm:main}
	There exists an open subset $\frU\subset C^\infty(\bbD^2,M)$ such that for every $f:\bbD^2\to M$ in $\frU$, the linearization $\frL_f$ of the operator $\frD$ at $f$ admits a tame right inverse.
\end{theorem}
The open subset $\frU$ consists of solutions to certain open first order differential relation in $J^1(\Sigma,M)$. We observe that the linearization $\frL_f$ of $\frD$ at $f\in\frU$, factors as a composition of two linear operators, one of which is elliptic. In particular, the inverse of $\frL_f$ in \autoref{thm:main} is not a differential operator unlike in the case when $f$ is a `regular' map in the sense of Gromov.

Applying the Nash-Hamilton Implicit Function Theorem, due to R. Hamilton (\cite{hamiltonNashMoser}), for smooth tame operators between tame Fr\'echet spaces we obtain the following result.
\begin{theorem}
	$\frD$ is locally invertible on $\frU$.
\end{theorem}

Next, we obtain jets of infinitesimal solutions to $\frD$. An application of the local invertibility then gives us the following local $h$-principle.

\begin{theorem}\label{thm:hPrinciple}
	Horizontal maps $\mathbb D^2\to M$ satisfy the local h-principle.
\end{theorem}

We also prove the existence of horizontal germs.

\begin{theorem}\label{thm:existenceOfGerms}
	There exist germs of $\calD$-horizontal submanifolds of dimension $2$.
\end{theorem}

The article is arranged as follows. In \autoref{sec:basics} we first discuss the preliminaries of fat distributions. In \autoref{sec:inversion} we prove our main result : the local invertibility of the nonlinear differential operator $\frD$. Then in \autoref{sec:hPrinicple} we derive the local h-principle for horizontal maps and prove the existence of germs of such maps. In order to make the article self-contained, we briefly outline in \autoref{sec:hamiltonIFT} the background and the statement of the Implicit Function Theorem for differential operators between Frechet spaces following \cite{hamiltonNashMoser}, as this has been crucially used in \autoref{sec:inversion}.

\section{Basic Notions}
\label{sec:basics}
In this section we recall the notion of fat distribution and then focus on distributions which are of corank $2$. 
\subsection{Bracket-Generating Distributions}
A distribution $\calD$ on a manifold $M$ is a subbundle of the tangent bundle $TM$. A distribution $\calD$ can equivalently be identified with its sheaf of sections $\Gamma\calD$. Hence by $X\in\calD$ we will mean that $X$ is a local section of $\calD$. We denote by $[\calD,\calD]$ the sheaf of all vector fields on $M$ which are obtained by taking Lie brackets of two sections in $\calD$. The sheaf $[\calD,\calD]$ need not be associated to a distribution, since it may fail to have constant rank. We define recursively, for $i\geq 1$,
$$\calD^{i+1} = \calD^i + [\calD,\calD^i], \quad \calD^1 = \calD.$$

\begin{defn}
	$\calD$ is said to be \emph{bracket generating} if, at each point $x\in M$, we have, $\calD^r|_x=T_x M$, for some positive integer $r$, possibly depending on $x$.
\end{defn}

\begin{example}\label{eg:contactDist}
	A corank 1 distribution $\xi\subset TM$ on $M$ is defined locally as the kernel of a 1-form $\alpha$ on $M$. The distribution $\xi$ is said to be a \emph{contact distribution} if $\alpha\wedge(d\alpha)^n\ne 0$. The property does not depend on the choice of defining 1-form $\alpha$. By the Darboux theorem (\cite{geigesBook}), $\alpha$ has a local normal form $\alpha = dz - \sum_{i=1}^n y_i dx_i$. It is then easy to check that, $$\xi \underset{loc.}{=} \Span \langle \partial_{y_i}, \partial_{x_i} + y_i \partial_z\rangle.$$
	Since $[\partial_{y_i},\partial_{x_i} + y_i \partial_z] = \partial_z$, we see that $TM =\xi + [\xi,\xi] = \xi^2$. Hence, contact distributions are bracket generating. 
\end{example}

In this article we will be interested in distributions which are $1$-step bracket generating, i.e, $TM=\calD^2$. If $\rk\calD=n$ and $\cork\calD=p$ then we say that $\calD$ is of type $(n,n+p)$. As observed, every contact distribution is of type $(2n,2n+1)$ for some $n\ge 1$.

\subsection{Fat Distributions}\label{sec:fatDistrbution}
We can associate to any distribution $\calD$, its curvature form $\Omega$ which is a $TM/\calD$-valued 2-form on $\calD$. The curvature form plays a very important role in subriemannian geometry.
\begin{defn} Given a distribution $\calD$, we have the quotient map $\lambda : TM\to TM/\calD$. The \textit{curvature form} of the distribution is a map $\Omega : \Lambda^2\calD \to TM/\calD$ defined as follows: 
	$$\Omega(X,Y) := -[X,Y]\mod \calD = -\lambda([X,Y]), \quad \text{for local sections $X,Y\in \calD$.}$$
\end{defn}

Observe that $\Omega$ is $C^\infty(M)$-linear. $\calD$ is $1$-step bracket-generating precisely when $\Omega$ is surjective. Now, there is a special class of bracket generating distributions, called fat. These can be defined in several equivalent ways as described below.
\begin{description}
	\item[Strong Bracket Generation ] A distribution $\calD$ on a manifold $M$ is called \emph{fat} (or \emph{strongly bracket generating}) at $x\in M$ if, for every nonzero vector $v\in\calD_x$, we have that $$T_x M = \calD_x + [V, \calD]_x,$$
	where $V$ is some (local) section of $\calD$ with $V_x=v$ and $[V,\calD]_x$ is the subspace of $T_xM$ defined as follows:
	$$[V,\calD]_x = \big\{[V,X]_x \big| X\in \calD\big\}$$
	The distribution is fat if it is fat at every point $x\in M$. 
	
	\item[Nondegeneracy of Curvature Form ] Suppose we identify $\big(TM/\calD\big)^*$ with the annihilator bundle $\calD^\perp\subset T^*M$ and define the \textit{dual curvature form} $\omega : \calD^{\perp}\to \Lambda^2\calD^*$ by $\omega(\alpha) = d\alpha|_\calD$. We say that $\calD$ is fat at $x\in M$ if and only if $\omega(\alpha)$ is a nondegenerate $2$-form on $\calD$, for each $0\ne \alpha\in\calD^{\perp}_x$. The distribution $\calD$ is fat if it is fat at every $x\in M$
\end{description}

\begin{remark}\label{rmk:fatAndSymplectic}
	Clearly, any fat distribution is $1$-step bracket generating. Another important consequence of fatness is that $d\alpha|_\calD$ is nondegenerate for \emph{any} non-vanishing, (local) 1-form $\alpha$ annihilating $\calD$. Indeed, for a corank $p$ distribution, the strongly bracket generating property implies (and is implied by) that for any $0\ne v\in\calD_x$, the map
	\begin{align*}
		\Phi_v:\calD_x &\to T_xM/{\calD_x}\cong \bbR^p\\
		u &\mapsto \big(\iota_v \omega_1(u),\ldots,\iota_v\omega_p(u)\big)
	\end{align*}
	is surjective where $\calD\underset{loc.}{=}\cap_{i=1}^p\ker\lambda^i$ and $\omega_i = d\lambda^i|_\calD$. This, in particular, implies that $\omega_i$ are non-degenerate for each $i=1,\ldots,p$. 
\end{remark}

\begin{example}
	A contact distribution is fat. In fact, contact distributions are precisely the fat distributions in the corank 1 case.
\end{example}

Since fatness is an open condition, fat distribution germs form an open set in the space of all germs of distributions of a fixed rank and corank. In general, fatness imposes strong numerical constraints on the rank and corank of the distribution.
\begin{theorem}[\cite{raynerFat,montTour}]\label{thm:fatNumericalConstraints}
	Suppose $\calD$ is a rank $k$-distribution on $M$ with $\dim M=n$. If $\calD$ is fat then the following numeric constraints hold.
	\begin{itemize}
		\item $k$ is divisible by $2$ and if $k < n-1$, then $k$ is divisible by $4$
		\item $k \ge (n-k)+1$
		\item The sphere $S^{k-1}$ admits $n-k$ linearly independent vector fields
	\end{itemize}
	Conversely, given any pair $(k,n)$ satisfying the above, there is a germ of fat distribution of type $(k,n)$.
\end{theorem}

\begin{remark}\label{rmk:fatNonGeneric}
	Though contact distributions happen to be generic in the corank $1$ situation, fatness in higher corank is not a generic property. For corank $\ge 3$, this follows easily from the Determinacy theorem (\cite[pg. 65]{montTour}). As for corank $2$, we can define the signature of a (germ of) distribution of type $(4,6)$ (\cite[pg. 92]{montTour}). We have two disjoint \emph{open} classes of distributions germs of type $(4,6)$ : the \emph{elliptic} and the \emph{hyperbolic} type (\cite{zanetGenericRank4}). The elliptic type corresponds to fat distribution germs, whereas a typical example of a hyperbolic type is given by the product of two contact structures of type $(2,3)$ as considered in \cite{dAmbraSubbundle,bandeContactPair}.
\end{remark}

It follows that the fat distributions are of the type $(4n, 4n + p)$ when $p>1$. We will now focus on $p=2$.

\subsection{Corank $2$ Fat Distributions}
Suppose $\calD$ is a corank $2$ fat distribution on $M$. For simplicity, let us first assume that $\calD$ is cotrivial. Hence there exist $1$-forms $\lambda^1,\lambda^2\in\Omega^1(M)$ such that the quotient map $\lambda:TM\to TM/\calD \cong M\times \bbR^2$ is given as, $\lambda=(\lambda^1,\lambda^2)$ and $$\calD=\ker\lambda^1\cap\ker\lambda^2.$$ Moreover, the curvature $2$-form $\Omega:\Lambda^2\calD\to TM/\calD\cong M\times\bbR^2$ is given by 
$$\Omega(X,Y) =\big(\omega_1(X,Y), \omega_2(X,Y)\big),$$
where $\omega_i=d\lambda^i|_\calD$. Since $\calD$ is fat, we have from \autoref{rmk:fatAndSymplectic} that $\omega_1, \omega_2$ are nondegenerate. Therefore, we can define an automorphism $A:\calD\to \calD$ by the following rule :
$$\omega_1(u, Av) = \omega_2(u,v), \forall u,v\in\calD.$$
Explicitly, we have, $$A=-I_{\omega_1}^{-1}\circ I_{\omega_2},$$
where $I_{\omega_i} : \calD\to\calD^*$ is the induced isomorphism $I_{\omega_i}(v) = \iota_v \omega_i$.

For any subspace $V\subset \calD_x$, denote the symplectic complement of $V$ with respect to $\omega_i$ by $V^{\perp_i}$, $i=1,2$.
$$V^{\perp_i} = \big\{w\in \calD_x| d\lambda^i(v,w) = 0,\forall v\in V\big\}.$$
It is easy to deduce that
$$V^{\perp_2} = \big(AV\big)^{\perp_1}, \quad V^{\perp_1} = A\big(V^{\perp_2}\big).$$

For a general $\cork 2$ fat distribution $\calD$, not necessarily cotrivializable, the automorphism $A$ can only be defined locally, since it depends on the choice of annihilating forms for $\calD$. We observe a criteria for fatness, in the corank $2$ situation.
\begin{prop}\label{prop:fatEigenValue}
	Suppose $\calD$ is a corank $2$ distribution on $M$ defined locally by a pair of 1-forms $\lambda^1,\lambda^2$. Then $\calD$ is fat if and only if the following conditions are satisfied:
	\begin{itemize}
		\item $\omega_i = d\lambda^i|_\calD$ is nondegenerate for $i=1,2$.
		\item The (local) automorphism $A:\calD\to \calD$ relating $\omega_1,\omega_2$ has no real eigenvalue.
	\end{itemize}
	
\end{prop}
\begin{proof}
	First suppose that $\calD$ is strongly bracket generating at $x$. This means that for any $0\ne v\in\calD_x$, the map
	\begin{align*}
		\Phi_v:\calD_x &\to T_xM/{\calD_x}\cong \bbR^2\\
		u &\mapsto \big(\iota_v \omega_1(u),\iota_v\omega_2(u)\big)
	\end{align*}
	is surjective. This, in particular, implies that $\omega_i$ are non-degenerate. On the other hand, if $\omega_1$ and $\omega_2$ are non-degenerate then  
	\begin{center}$\Phi_v$ is onto $\Leftrightarrow$ $\cork(v^{\perp_1}\cap v^{\perp_2})=2$\end{center}
	Since $v^{\perp_2}=(Av)^{\perp_1}$, this is equivalent to $A$ having no real eigenvalue. This completes the proof.\end{proof}

\begin{example}\label{exmp:contactHolomorphic}
	A holomorphic $1$-form $\Theta$ on a complex manifold $M$ with $\dim_\bbC M = 2n+1$, is called \emph{holomorphic contact} if it satisfies $\Theta \wedge d\Theta^n \ne 0$. By the holomorphic Darboux theorem (\cite{forstnericHoloLegendrianCurves}), we have holomorphic coordinates $(z,x_1,\ldots,x_n,y_1,\ldots,y_n)$ on $M$ such that holomorphic contact form is given as, $\Theta \underset{loc.}{=} dz - \sum_{j=1}^n y_j dx_j$. 
	If we identify, $\bbC^{2n+1}$ with $\bbR^{4n+2}$ and write $z=z_1 + \iota z_2,\, x_j = x_{j1} + \iota x_{j2}, \, y_j = y_{j1} + \iota y_{j2}$, then $\Theta$ can be expressed as $\Theta = \lambda^1 + \iota\lambda^2$, where 
	$$\lambda^1 = dz_1 - \sum_{j=1}^n \big(y_{j1} dx_{j1} - y_{j2}dx_{j2}\big),\qquad \lambda^2 = dz_2 - \sum_{j=1}^n \big(y_{j2} dx_{j1} + y_{j1}dx_{j2}\big).$$
	This gives us a corank $2$ distribution $\calD = \ker \lambda^1 \cap \ker\lambda^2\subset TM$, which is canonically isomorphic to the holomorphic contact subbundle $\ker \Theta \subset T^{(1,0)}M$. We can explicitly write down a local frame $\calD=\Span \langle X_{j1}, X_{j2}, Y_{j1}, Y_{j2}\rangle$, where
	$$X_{j1} = \partial_{x_{j1}} + y_{j1}\partial_{z_1} + y_{j2}\partial_{z_2}, \quad X_{j2} = \partial_{x_{j2}} - y_{j2}\partial_{z_1} + y_{j1}\partial_{z_2}, \quad Y_{j1} = \partial_{y_{j1}},\quad Y_{j2} = \partial_{y_{j2}}.$$
	Then the connecting automorphism $A:\calD\to\calD$ defined by $d\lambda^1(u,Av) = d\lambda^2(u,v)$ for each  $u,v\in\calD$, satisfies the following equations :
	$$AX_{j1} = -X_{j2}, \quad AX_{j2} = X_{j1},\qquad AY_{j1} = -Y_{j2}, \quad AY_{j2} = Y_{j1}.$$
	In particular, we have $A^2 = -Id$, that is, $A$ induces a complex structure on $\calD$. Since $A$ has no real eigenvalue, by \autoref{prop:fatEigenValue} the distribution $\calD$ is fat.
\end{example}

Furthermore, the tangent bundle of $M$ splits as the direct sum $TM=\calD\oplus \Span \langle Z_1,Z_2\rangle$, where $Z_i = \partial_{z_i}$ are two vector fields satisfying the relations below : 
$$[Z_1,Z_2] = 0, \quad \lambda^i(Z_j) = \delta_{ij}, \quad \iota_{Z_i}d\lambda^j|_\calD = 0, \qquad i,j=1,2.$$
Motivated by this, we consider the following.

\begin{defn}\label{defn:fatWithReeb}
	A corank $2$ distribution $\calD$ on $M$ is said to admit (local) Reeb directions $Z_1,Z_2$, if $\calD\underset{loc.}{=} \ker\lambda^1\cap\ker\lambda^2$ and $TM=\calD \oplus \Span \langle Z_1,Z_2\rangle$ such that,
	\begin{enumerate}[\qquad (a)]
		\item \label{defn:fatWithReeb:a} $\lambda^1(Z_1) = 1, \; \lambda^1(Z_2) = 0$
		\item \label{defn:fatWithReeb:b} $\lambda^2(Z_1) = 0, \; \lambda^2(Z_2) = 1$
		\item \label{defn:fatWithReeb:c} $\iota_{Z_i} d\lambda^j|_\calD = 0$ for $i,j=1,2$
		\item \label{defn:fatWithReeb:d} $[Z_1,Z_2]=0$
	\end{enumerate}
\end{defn}

As observed in \autoref{exmp:contactHolomorphic}, the real distribution associated to the holomorphic contact structure, admits (local) Reeb directions.

Given any corank $2$ fat distribution $\calD$ on a manifold $M$ of dimension $4n+2$, we may find (\cite{geCharacteristicClass}) a coordinate system $(x_1,\ldots,x_{4n},z_1,z_2)$ and $1$-forms $$\lambda^i = dz_i - \sum_{j,k} \Gamma^i_{jk} x_j dx_k + R_i,\quad i=1,2,$$
such that $\calD\underset{loc.}{=}\ker\lambda^1 \cap \ker\lambda^2$. Here $R_i=\sum_{j=1}^2 f_{ij} dz_j + \sum_{j=1}^{4n} g_{ij}dx_j$ is a $1$-form such that, $f_{ij}, g_{ij}\in O(|x|^2 + |z|^2)$, and $\{\Gamma^i_{jk}\}$ constitute the structure constants of some nilpotent Lie algebra, known as the nilpotentization (\cite{montTour, tanakaDiffSystem}), associated to the distribution $\calD$. In particular $\Gamma^i_{jk} = -\Gamma^i_{kj}$. Observe that, if we take $f_{ij}=0$ and $g_{ij}$ to be functions of $x_k$'s only, then any such tuple of forms $(\lambda^1,\lambda^2)$ above gives a corank $2$ distribution, which admits local Reeb directions $ (\partial_{z_1},\partial_{z_2})$.

From the classification results of \cite{nilpoten6DimLieAlgebra}, we see that the only possible Lie algebra that can arise as the nilpotentization of a corank $2$ fat distribution on a $6$ dimensional manifold is the complex Heisenberg Lie algebra.

\begin{question}\label{question:uniqueFat}
	Is every (germ of) corank $2$ fat distribution on $\bbR^6$, which admits local Reeb directions, diffeomorphic to the germ of the distribution associated to a holomorphic contact structure as in \autoref{exmp:contactHolomorphic}?
\end{question}

For a general corank 2 fat distribution, the answer is clearly no. From a result of Montgomery (\cite{montGeneric}), it follows that a \emph{generic} distribution germ of type $(4,6)$ cannot admit a local frame generating a finite dimensional Lie algebra. This differs from a holomorphic contact  distribution, which does admit a local frame generating the complex Heisenberg Lie algebra (see \autoref{exmp:contactHolomorphic}). Since the set of germs of fat distributions of type $(4,6)$ is open, there are plenty of fat distributions, non-diffeomorphic to the contact holomorphic one. But it is not clear whether any of these fat distributions admits (local) Reeb directions.

Note that if the answer to the above question is in the affirmative, we can characterize germs of horizontal immersions, given by the $1$-jet prolongation of holomorphic maps $\bbC \to \bbC$. But we suspect that the answer is negative as the $1$-forms obtained in \cite{geCharacteristicClass} indicates the presence of function moduli. Let us now study the question of horizontal immersions.

\section{Horizontal immersions in a corank 2 fat distribution}
\label{sec:inversion}
Suppose $\calD$ is a corank $2$ fat distribution on $M$ defined by a pair of $1$-forms $\lambda^1,\lambda^2$. Hence $\omega_i=d\lambda^i|_{\calD}$ are non-degenerate and the connecting homomorphism $A:\calD\to \calD$ defined by 
$$\omega_2(u,v)=\omega_1(u,A_x v), \quad \forall u, v\in\calD_x, \; x\in M$$
has no real eigenvalue. We further assume that \textit{the distribution $\calD$ admits local Reeb directions}.

Now for a fixed manifold $\Sigma$, consider the partial differential operator,
\begin{align*}
	\frD : C^\infty(\Sigma,M) &\to \Omega^1(\Sigma,\bbR^2)\\
	f &\mapsto \big(f^*\lambda^1, f^*\lambda^2\big)
\end{align*}
The $C^\infty$-solutions of $\frD(f)=0$ are precisely the $\calD$-horizontal maps since the derivative of $f$ maps $T\Sigma$ into $\calD$. Furthermore, horizontality implies that $f^*d\lambda^1=0=f^*d\lambda^2$; hence $df_x:T_x\Sigma\to \calD_{f(x)}$ is an \emph{isotropic} map with respect to both the forms $\omega_1$ and $\omega_2$ on $\calD$ for every $x\in \Sigma$. Now, linearizing $\frD$ at an $f\in C^\infty(\Sigma,M)$ we get the linear differential operator $\frL_f$ as follows :
\begin{align*}
	\frL_f : \Gamma f^*TM &\to \Omega^1(\Sigma,\bbR^2)\\
	\partial &\mapsto \Big(d\big(\lambda^i\circ\partial\big) + f^*\iota_\partial d\lambda^i\Big)_{i=1,2}
\end{align*}
Restricting $\frL_f$ to $\Gamma f^*\calD$ we have,
\begin{align*}
	\calL_f : \Gamma f^*\calD &\to \Omega^1(\Sigma,\bbR^2)\\
	\partial &\mapsto \big(f^*\iota_\partial d\lambda^i\big)_{i=1,2}
\end{align*}
Observe that $\calL_f$ is a $C^\infty(\Sigma)$-linear map and hence is induced by a bundle map $f^*\calD\to T^*\Sigma\otimes\bbR^2$.

An horizontal immersion $f$ is said to be \emph{regular} if this bundle map $\calL_f$ is surjective; this is referred to as $\Omega$-regularity in \cite{gromovCCMetric}. One then gets that the operator $\frD$ is infinitesimally invertible over the set of regular maps and an appeal to Gromov's general theorems (\cite{gromovBook}). But for such a map $f$ to exist, i.e, for the existence of a regular, common isotropic subspace $V\subset \calD$, we must have the inequality (\cite{gromovCCMetric})
$$\rk\calD - \dim\Sigma  \ge 2\dim\Sigma,$$
as common isotropic subspaces are necessarily in the kernel of the map $\calL_f$. We will be focusing on $\calD$-horizontal immersions of discs $\bbD^2$ in $6$-dimensional manifold, where $\rk\calD=4$. Clearly, $\rk\calD = 4 \not\ge 6 = 3.2=3\dim\Sigma$. Hence there is no possibility of an \emph{regular} horizontal map $\bbD^2\to M$ to exist and so Gromov's method does not apply directly.

\subsection{Inversion of $\frL_f$ at $\calD$-horizontal Immersions}
We now denote $\Sigma=\bbD^2$ and $M=\bbR^6$. Suppose $\calD=\ker\lambda^1\cap\ker\lambda^2$ is a given corank $2$ fat distribution, which admits local Reeb directions. Since $\Sigma$ is a compact manifold with boundary, we have (see \autoref{exmp:frechetSpaces}),
\begin{obs}
	The spaces $\Gamma(f^*TM)$ and $\Omega^1(\Sigma,\bbR^2)$ are tame Fr\'echet spaces.
\end{obs}
As before we have the linearization map, $\frL_f:\Gamma f^*\calD\to \Omega^1(\Sigma,\bbR^2)$. Since $\frL_f$ is a linear partial differential operator of order $1$, we have (see \autoref{exmp:tameOperators}),
\begin{obs}\label{obs:linearizationTame}
	$\frL_f$ is a tame linear map of order $1$.
\end{obs}

This sets the problem into the framework of differential operators between Fr\'echet spaces for studying the existence of local inversion. We refer to the appendix (\autoref{sec:hamiltonIFT}) for relevant details. 
We first prove the following result.
\begin{prop}\label{prop:horizontalInversion}
	If $f$ is a smooth horizontal immersion, then $\frL_f$ admits a tame inverse $\frM_f$. 
\end{prop}
Note that we are assuming the existence of $\calD$-horizontal immersions in the above proposition. In fact, in the next section, we shall prove the inversion for an \emph{open} set of maps (\autoref{thm:openInversion}). Let us first prove the following.
\begin{lemma}\label{lemma:invariant}
	If $V\subset\calD_x$ is common isotropic with respect to $\omega_i=d\lambda^i|_\calD$ and $\dim V = 2$, then $V=AV$.
\end{lemma}
\begin{proof}
	Since $V$ is common isotropic,	
	$$V\subset V^{\perp_1}\cap V^{\perp_2} = (V+AV)^{\perp_1} \;\Rightarrow\; \dim(V+AV)^{\perp} \ge \dim V = 2$$
	and so, $\dim (V+AV) \le \dim \calD_x - 2 = 2$. On the other hand, $\dim(V+AV) \ge \dim V = 2$. Hence, $\dim(V+AV) = 2 = \dim V$, which is only possible if $V = AV$.
\end{proof}

\begin{prop}\label{prop:uniqueSolution}
	If $f$ is a smooth $\calD$-horizontal immersion, given any $(P,Q)\in\Omega^1(\Sigma,\bbR^2)$, the equation $\frL_f(\partial) = (P,Q)$ admits a unique solution $\partial = \frM_f(P,Q)$, subject to a boundary condition. The process of obtaining the solution depends on a choice of complex structure $J$ on $\calD$.
\end{prop}
\begin{proof}
	First, choose an almost complex structure $J$ on $\calD$, \emph{compatible} with $\omega_1 = d\lambda^1|_\calD$, i.e, the assignment $(u,v) \mapsto \omega_1(u, Jv)$ is a nondegenerate \emph{symmetric} form. Such a $J$ always exists (\cite[pg. 86]{daSilvaBook}). Clearly, $J \ne A$. 
	
	Since $f$ is $\calD$-horizontal we have, $$f^*\lambda^i = 0\Rightarrow f^*d\lambda^i = 0.$$
	Thus, for $\sigma\in\Sigma$, $\im df_\sigma$ is common isotropic with respect to both $\omega_i=d\lambda^i|_\calD$. In particular, $\im df_\sigma$ is $J$-totally real, since $J$ is $\omega_1$-compatible. Also since $f$ is an immersion, $\dim\im df_\sigma=2$. Then by \autoref{lemma:invariant}, we have that $$A(\im df_\sigma)=\im df_\sigma, \quad \text{for $\sigma\in\Sigma$.}$$
	Let us denote, $X=f_*(\partial_x), Y=f_*(\partial_y)$, where $\partial_x,\partial_y$ are the coordinate vector fields on $\Sigma=\bbD^2$. We thus have 
	$$\Span \langle AX, AY\rangle = \Span \langle X, Y\rangle.$$
	Hence, $A$ restricts to an automorphism on $\Span \langle X,Y\rangle$ : 
	$$A_0 = A|_{\Span \langle X, Y\rangle}.$$
	Let us write, \begin{equation*}
		AX = p X + q Y, \quad AY = r X + s Y \tag{$*$} \label{eqn:valueOfAXAY}
	\end{equation*}
	for some functions $p,q,r,s\in C^\infty(\Sigma)$. Then we have that $A_0=\begin{pmatrix}p & q \\ r & s\end{pmatrix}$ with respect to the basis $(X,Y)$. Since $A$ has no real eigenvalue, $A_0$ also has no real eigenvalue. This means that the characteristic polynomial $$\lambda^2 - (p+s)\lambda + (ps - qr)$$ of $A_0$ has negative discriminant, i.e., 
	$$(p+s)^2 - 4(ps - qr) = (p-s)^2 + 4qr < 0.$$
	
	Now let us consider the equation $$\frL_f(\partial) = (P,Q),$$ where $P,Q\in\Omega^1(\Sigma)$. We write $$\partial= \partial_0 + a Z_1 + b Z_2,$$
	where $\partial_0\in f^*\calD$ and $Z_1,Z_2$ are the Reeb directions associated to $(\lambda^1,\lambda^2)$, pulled back along $f$. We then have,
	$$\frL_f(\partial) = \Big(da + f^*\iota_{\partial_0} d\lambda^1, db + f^*\iota_{\partial_0} d\lambda^2\Big)$$
	Also let us write $$P=P_1 dx + P_2 dy, \quad Q = Q_1 dx + Q_2 dy.$$
	Evaluating both sides on $\partial_x,\partial_y$ and using properties (\ref{defn:fatWithReeb:a}), (\ref{defn:fatWithReeb:b}), (\ref{defn:fatWithReeb:c}) of \autoref{defn:fatWithReeb}, we have the system,
	\begin{align}
		\left\{\quad
		\begin{aligned}\label{eqn:horizontalOriginalSystemFirstSet}
			\partial_x a + d\lambda^1(\partial_0, X) &= P_1\\
			\partial_y a + d\lambda^1(\partial_0, Y) &= P_2\end{aligned}\right.\\
		\left\{\quad
		\begin{aligned}\label{eqn:horizontalOriginalSystemSecondSet}
			\partial_x b + d\lambda^2(\partial_0, X) &= Q_1\\
			\partial_y b + d\lambda^2(\partial_0, Y) &= Q_2
		\end{aligned}\right.
	\end{align}
	
	Now using (\ref{eqn:valueOfAXAY}) we have, $$d\lambda^2(\partial_0,X) = d\lambda^1(\partial_0, AX) = p\, d\lambda^1(\partial_0, X) + q\, d\lambda^1(\partial_0, Y)$$
	$$d\lambda^2(\partial_0,Y) = d\lambda^1(\partial_0, AY) = r\, d\lambda^1(\partial_0, X) + s\, d\lambda^1(\partial_0, Y)$$
	This transforms \eqref{eqn:horizontalOriginalSystemSecondSet} into the following system of PDEs :
	\begin{align}
		\left\{\quad
		\begin{aligned}
			\partial_x b + p\, d\lambda^1(\partial_0, X) + q\, d\lambda^1(\partial_0, Y) &= Q_1\\
			\partial_y b + r\, d\lambda^1(\partial_0, X) + s\, d\lambda^1(\partial_0, Y) &= Q_2
		\end{aligned}
		\right.\tag{$2^\prime$}\label{eqn:horizontalChangedSystemSecondSet}
	\end{align}
	Using \eqref{eqn:horizontalOriginalSystemFirstSet} we eliminate $\partial_0$ from \eqref{eqn:horizontalChangedSystemSecondSet} and get
	\begin{align}
		\left\{\quad
		\begin{aligned}\label{eqn:horizontalAfterElimination}
			\partial_x b - p\partial_x a - q\partial_y a &= Q_1 - p P_1 - q P_2\\
			\partial_y b - r\partial_x a - s\partial_y a &= Q_2 - r P_1 - s P_2
		\end{aligned}
		\right.\tag{$2^{\prime\prime}$}
	\end{align}
	Since $(p-s)^2 + 4qr < 0$, the system of PDEs given by \eqref{eqn:horizontalAfterElimination} is elliptic. Hence, the Dirichlet problem \eqref{eqn:horizontalAfterElimination} with the boundary condition 
	\begin{align}
		a|_{\partial\Sigma} = a_0, \quad b|_{\partial\Sigma} = b_0,\label{eqn:boundary}
	\end{align} 
	will have a unique solution  
	$$(a,b) = M_f(P,Q, a_0, b_0).$$
	
	Now consider an auxiliary system of equations :
	\begin{align}
		\left\{\quad
		\begin{aligned}d\lambda^1(\partial_0, JX) &= 0\\
			d\lambda^1(\partial_0, JY) &= 0
		\end{aligned}
		\right.\label{eqn:horizontalAdditional}
	\end{align}
	Then using the solution $(a,b)=M_f(P,Q,a_0,b_0)$, we get from \eqref{eqn:horizontalOriginalSystemFirstSet}, \eqref{eqn:horizontalAdditional}, the system
	\begin{align}
		\left\{\quad
		\begin{aligned}
			d\lambda^1(\partial_0, X) &= P_1 - \partial_x a\\
			d\lambda^1(\partial_0, Y) &= P_2 - \partial_y a\\
			d\lambda^1(\partial_0, JX) &= 0\\
			d\lambda^1(\partial_0, JY) &= 0
		\end{aligned}
		\right.\label{eqn:horizontalLinear}
	\end{align}
	Since $\im df_\sigma$ is $J$-totally real, $(X,Y,JX,JY)$ is a local framing of $\calD$, and since $d\lambda^1|_\calD$ is nondegenerate, \eqref{eqn:horizontalLinear} can be uniquely solved for $\partial_0$. Thus, $\frL_f(\partial)=(P,Q)$ has a unique solution   
	$$\partial = \frM_f(P,Q, a_0, b_0)$$
	subject to satisfying the auxiliary system \eqref{eqn:horizontalAdditional} and the boundary condition \eqref{eqn:boundary}.
\end{proof}

\begin{remark}
	It can be easily seen from \autoref{exmp:contactHolomorphic} that for our model case $(M,\calD)$ of holomorphic contact structure, we have $A = -J|_\calD$, where $J$ is the (integrable) almost complex structure on $M$. \autoref{lemma:invariant} can then be interpreted as follows : common isotropic $2$-subspaces of $\calD$ are complex subspaces. In particular, the left hand side of (\ref{eqn:horizontalAfterElimination}) can then be compared to the usual Cauchy-Riemann equations for the tuple of functions $(a,b)$ on $\Sigma$.
\end{remark}

We can now prove \autoref{prop:horizontalInversion}
\begin{proof}[Proof of \autoref{prop:horizontalInversion}]
	From \autoref{prop:uniqueSolution} we have that $\frL_f$ admits unique solution $\frM_f$, whenever $f$ is a $\calD$-horizontal immersion. As in \ref{prop:uniqueSolution}, $M_f$ is obtained as a solution to a Dirichlet problem and hence it is tame (see \autoref{exmp:tameOperators} (\ref{exmp:tameOperators:2})). Then $\frM_f$ is obtained from $M_f$ by solving a linear system, which is again tame. Hence the inverse $\frM_f$ is tame, as composition of tame maps is tame.
\end{proof}

\begin{remark}
	In fact, the operator $\frM_f$ above is tame of degree $1$. Indeed, the proof of tameness for elliptic boundary value problems (\cite[pg. 161]{hamiltonNashMoser}) suggests that $M_f$ is tame of degree $0$. Next, to get $\frM_f$ from $M_f$, the linear system (\ref{eqn:horizontalLinear}) involves taking first order differentials and hence it is tame of degree $1$. Thus, $\frM_f$ is tame of degree $1$.
\end{remark}

\subsection{Local Inversion of $\frD$}
From \autoref{prop:horizontalInversion} we see that the linearization $\frL_f$ admits a right inverse $\frM_f$, provided $f$ is a $\calD$-horizontal immersion. But in order to apply the Implicit Function Theorem due to Hamilton (\autoref{thm:hamiltonNashMoser}), we need to show that there is an \emph{open} set of maps $\frU\subset C^\infty(\Sigma, M)$ such that the family $\{\frL_f\;|\;f\in\frU\}$ admits a smooth tame inverse. We now identify this set $\frU$.\\

We first restrict ourselves to a collection $\frU_0$ of maps $f:\Sigma\to M$ satisfying the following conditions :
\begin{itemize}
	\item $f$ is an immersion, and
	\item $\im df$ is transverse to $\Span\langle Z_1, Z_2\rangle$.
\end{itemize}
This collection $\frU_0\subset C^\infty(\Sigma,M)$ is clearly open, since it is defined by open conditions. Now we have a canonical projection $$\pi_\calD : TM = \calD\oplus \Span\langle Z_1,Z_2\rangle\to \calD.$$
For any $f\in\frU_0$ we see that the image $\pi_\calD(\im df)$ has dimension $2$ at each point of $\Sigma$. Let us choose an almost complex structure $J:\calD\to\calD$, compatible with $d\lambda^1|_\calD$, as in \autoref{prop:uniqueSolution}. Then the set $$\Big\{(X,Y)\in \mathrm{Fr}_2\calD \;\Big|\; \text{$V=\Span\langle X,Y\rangle$ is $J$-totally real}\Big\}$$
is open in the $2$-frame bundle $\mathrm{Fr}_2\calD$, since the totally real condition $V\cap JV=0$ is open. For any such tuple $(X,Y)$ we have the framing $(X, Y, JX, JY)$ of $\calD$ and we can write $$A=\begin{pmatrix}A_{11} & A_{12}\\ A_{21} & A_{22}\end{pmatrix}$$
with respect to this basis. Let $\calO_x\subset \mathrm{Fr}_2\calD_x$ be the set of those $(X,Y) \in \mathrm{Fr}_2\calD_x$ such that,
\begin{itemize}
	\item $V = \Span\langle X,Y \rangle$ is $J$-totally real, and
	\item The matrix $A_{11}$ as above is negative definite.
\end{itemize}
Since both are open conditions, we see that $\calO_x$ is open in $\mathrm{Fr}_2\calD_x$.\\

We now define,
\begin{defn}\label{defn:admissibleMaps}
	A map $f:\Sigma\to M$ is said to be \textit{admissible} if it satisfies the following.
	\begin{itemize}
		\item $f\in\frU_0$, i.e, $f$ is an immersion with $\im df\pitchfork \Span \langle Z_1,Z_2\rangle$.
		\item $\im df_\sigma = \Span \langle f_*\partial_x, f_*\partial_y\rangle \in \pi_\calD^{-1}\big(\calO_{f(\sigma)}\big)$ for each $\sigma\in\Sigma$.
	\end{itemize}
	Denote by $\frU\subset C^\infty(\Sigma,M)$ the set of admissible maps.
\end{defn}
In fact we have defined an \emph{open} relation $\calA\subset J^1(\Sigma,M)$ such that $\frU$ is exactly the smooth holonomic solutions of $\calA$. Since $\calA$ is an open relation, we have that $\frU$ is open in $C^\infty(\Sigma,M)$. It is apparent that any $\calD$-horizontal immersion is admissible. We now prove the following.

\begin{theorem}\label{thm:openInversion}
	The linearization $\frL_f$ admits a smooth tame inverse $\frM_f$ for every $f\in\frU$.
\end{theorem}
\begin{proof}
	Suppose $f\in\frU$. We have $\im df = \Span \langle f_*\partial_x, f_*\partial_y\rangle$. Let us write $$f_*\partial_x = X + a_1 Z_1 + a_2 Z_2, \quad f_*\partial_y = Y + b_1 Z_1 + b_2 Z_2$$
	where $X=\pi_\calD(f_*\partial_x), Y=\pi_\calD(f_*\partial_y)$. By assumption $(X,Y)\in Fr_2\calD$ so that, $(X,Y,JX,JY)$ is a frame of $\calD$. Hence we can write
	\begin{align}
		\left\{\quad
		\begin{aligned}
			AX &= p X + q Y + p^\prime JX + q^\prime JY\\
			AY &= r X + s Y + r^\prime JX + s^\prime JY
		\end{aligned}
		\right.\label{eqn:generalInvariant}
	\end{align}
	The matrix of $A$ has the form $$\begin{pmatrix}
		p & r & * & *\\
		q & s & * & *\\
		p^\prime & r^\prime & * & *\\
		q^\prime & s^\prime & * & *
	\end{pmatrix}$$
	and by the hypothesis on $\frU$, $A_{11}=\begin{pmatrix}p & q \\ r & s\end{pmatrix}$ is negative definite, which is equivalent to $$(p-s)^2 + 4qr < 0.$$
	
	Now, we wish to solve $\frL_f(\partial) = (P,Q)$, as we did in \autoref{prop:uniqueSolution}, where
	\begin{align*}
		\frL_f : \Gamma f^*TM &\to \Omega^1(\Sigma,\bbR^2)\\
		\partial &\mapsto \Big(d\big(\lambda^i\circ\partial\big) + f^*\iota_\partial d\lambda^i\Big)_{i=1,2}
	\end{align*}
	Let $\partial = \partial_0 + a Z_1 + b Z_2$, where $\partial_0\in f^*\calD$. Since $[Z_1,Z_2]=0$ (by (\ref{defn:fatWithReeb:d}) of \autoref{defn:fatWithReeb}), we have, $$d\lambda^1(Z_1,Z_2) = Z_1(\lambda^1(Z_2)) - Z_2(\lambda^1(Z_1)) - \lambda^1([Z_1,Z_2]) = Z_1(0) - Z_2(1) - \lambda^1(0) = 0.$$
	and similarly, $d\lambda^2(Z_1,Z_2) = 0$. Hence, $$d\lambda^1(\partial, f_*\partial_x) = d\lambda^1(\partial_0 + a Z_1 + b Z_2, X + a_1 Z_1 + a_2 Z_2) = d\lambda^1(\partial_0, X)$$
	and similarly the remaining ones. Thus, we get a system as before :
	\begin{align}
		&\left\{\quad
		\begin{aligned}
			\partial_x a + d\lambda^1(\partial_0, X) &= P_1\\
			\partial_y a + d\lambda^1(\partial_0, Y) &= P_2\\
		\end{aligned}
		\right.\label{eqn:generalSystemFirstSet}\\
		&\left\{\quad
		\begin{aligned}
			\partial_x b + d\lambda^1(\partial_0, AX) &= Q_1\\
			\partial_y b + d\lambda^1(\partial_0, AY) &= Q_2
		\end{aligned}
		\right.\label{eqn:generalSystemSecondSet}
	\end{align}
	We add the linear equations 
	\begin{align}
		d\lambda^1(\partial_0, JX) = 0 = d\lambda^1(\partial_0, JY) \label{eqn:generalSystemAddition}
	\end{align}
	to \eqref{eqn:generalSystemFirstSet},\eqref{eqn:generalSystemSecondSet}. Then using \eqref{eqn:generalInvariant} and \eqref{eqn:generalSystemAddition}, the system \eqref{eqn:generalSystemSecondSet} becomes
	\begin{align}
		\left\{\quad
		\begin{aligned}
			\partial_x b + p\, d\lambda^1(\partial_0, X) + q\, d\lambda^1(\partial_0, Y) &= Q_1\\
			\partial_y b + r\, d\lambda^1(\partial_0, X) + s\, d\lambda^1(\partial_0, Y) &= Q_2
		\end{aligned}
		\right.\label{eqn:generalSystemSecondSetChanged}\tag{$8^\prime$}
	\end{align}
	Using \eqref{eqn:generalSystemFirstSet} we can eliminate $\partial_0$ in \eqref{eqn:generalSystemSecondSetChanged} and get
	\begin{align}
		\left\{\quad
		\begin{aligned}
			\partial_x b - p\partial_x a - q\partial_y a &= Q_1 - p P_1 - q P_2\\
			\partial_y b - r\partial_x a - s\partial_y a &= Q_2 - r P_1 - s P_2
		\end{aligned}
		\right.\label{eqn:generalSystemSecondSetEliminated}\tag{$8^{\prime\prime}$}
	\end{align}
	
	Since $(p-s)^2 + 4qr < 0$, we have that \eqref{eqn:generalSystemSecondSetEliminated} is elliptic. Hence given any arbitrary boundary condition $a|_{\partial\Sigma}=a_0,\; b|_{\partial\Sigma}=b_0$, we have the unique solution $$(a,b) = M_f(P,Q,a_0,b_0).$$
	Then as done in \autoref{prop:uniqueSolution}, we obtain unique solution $$\partial=\frM_f(P,Q,a_0,b_0)$$
	to the system given by  \eqref{eqn:generalSystemFirstSet}, \eqref{eqn:generalSystemSecondSet} and \eqref{eqn:generalSystemAddition}.
	Thus whenever $f\in\frU$, we have a solution $\frM_f$ for the linearized equation $\frL_f=(P,Q)$.
	As argued in the proof of \autoref{prop:horizontalInversion}, both $\frL_f$ and $\frM_f$ are tame operators.
\end{proof}

Since $\frL_f$ is surjective for every $f\in\frU$ and the family of right inverses $\frM:\frU\times \Omega^1(\Sigma,\bbR^2)\to C^\infty(\Sigma,M)$ is a smooth tame map we obtain the following by \autoref{thm:hamiltonNashMoser}.
\begin{theorem}
	The operator $\frD$ restricted to $\frU$ is locally right invertible. Given any $f_0\in\frU$, there exists an open neighborhood $U$ of $f_0$ and a smooth tame map $\frD_{f_0}^{-1}:\frD(U)\to U$ such that $\frD\circ \frD_{f_0}^{-1}=\textrm{Id}$.
\end{theorem}

The proof of the Implicit Function Theorem, in fact, implies that there exists a positive integer $r_0$ such that the following holds true.
\begin{theorem}\label{thm:solution} 
	Let $f_0\in\frU$ and $g_0 = \frD(f_0)$. Let $\epsilon>0$ be any positive number. Then there exists a $\delta > 0$ and an integer $r_0$, such that for $\alpha \ge r_0$ and for every $g\in\Omega^1(\Sigma,\bbR^2)$ with $|g|_{\alpha} < \delta$ there is an $f = \frD_{f_0}^{-1}(g_0 + g)\in\frU$ satisfying the following conditions :
	$$\frD(f) = g_0 + g \quad \text{ and } \quad |f - f_0|_{\alpha+2} < \epsilon.$$
\end{theorem}

\section{Existence of Horizontal Germs and the Local $h$-Principle}
\label{sec:hPrinicple}
Since we are only interested in germs, without loss of generality, we assume that $M=\bbR^6$ and $\Sigma=\bbR^2$. Suppose, we have a corank $2$ fat distribution $\calD$ on $M$, which admits Reeb directions (\autoref{defn:fatWithReeb}). Consider the relation $\calA\subset J^1(\Sigma, M)$, as in \autoref{sec:inversion}, such that the set of admissible maps (\autoref{defn:admissibleMaps}) $\frU$ are exactly the smooth holonomic sections of $\calA$, i.e, we have $\frU= \Sol\calR$. We have shown that the operator
\begin{align*}
	\frD:\frU\subset C^\infty(\Sigma, M) &\to \Omega^1(\Sigma,\bbR^2)\\
	f &\mapsto \big(f^*\lambda^1, f^*\lambda^2\big)
\end{align*}
is locally invertible over $\frU$.\\

Now, following Gromov (\cite{gromovBook}), we can get the (parametric) local h-principle. One crucial thing to observe is that the inversion of $\frD$ as we have obtained, does not conform to the notion of \emph{locality} as considered by Gromov (\cite[pg. 117-118]{gromovBook}). But we observe that the proof of the local $h$-principle goes through, without the locality property of $\calD^{-1}$. For the sake of completeness, we reproduce the proof following Gromov.

\begin{defn}(\cite[pg. 118]{gromovBook})
	A germ $f:\Sigma\to M$ at $\sigma\in \Sigma$ is called an \emph{infinitesimal solution} of order $\alpha$ of $\frD(f)=0$ if $$j^\alpha_{\frD(f)}(\sigma) = 0$$ i.e, the section $\frD(f)$ has zero $\alpha^{\text{th}}$-jet at the point $\sigma$.
\end{defn}

Observe that, since $\frD$ has order $1$, the property that $f$ is an infinitesimal solution of order $\alpha$, only depends on the jet $j^{\alpha + 1}_f(\sigma)$. Consider the relation $\calR_\alpha = \calR_\alpha(\frD,0,\calA) \subset J^{r+1}(\Sigma,M)$ consisting of jets $j^{\alpha+1}_f(\sigma)$ represented by $C^\infty$ germs $f:\Sigma\to M$ at $\sigma$, so that
$$j^\alpha_{\frD(f)}(\sigma) = 0 \quad\text{and}\quad j^{\alpha+1}_f(\sigma) \in\calA.$$
Then one sees that $C^{\alpha + 1}$ holonomic sections of $\calR_\alpha$ are exactly the admissible $C^{\alpha+1}$-solutions of $\frD = 0$. In particular, for $\alpha\ge 0$, the $C^\infty$-solutions $\calR_\alpha$ are all same, namely the $C^\infty$-solutions of $\frD = 0$ which are admissible, i.e, we have $$\text{$\Sol\calR_\alpha$ is the set of $\calD$-horizontal immersions, for any $\alpha\ge 0$.}$$
We then prove the following.

\begin{theorem}\label{thm:existenceOfLocalSolution}
	If $\alpha$ is sufficiently large, then for any jet  $j^\alpha_f(\sigma)\in\calR_\alpha$, represented by some $f : \Op(\sigma) \to M$, we have a homotopy $f_t : \Op(\sigma)\to M$, such that $f_0 = f$ on some $\Op(\sigma)$ and $f_1$ is a $\calD$-horizontal admissible solution, i.e, $\frD(f_1) = 0$. Furthermore, the jet $j^{\alpha+1}_{f_t}(\sigma)$ belongs to $\calR_\alpha$, for all $t\in [0,1]$.
\end{theorem}
\begin{proof}
	Suppose $f$ is defined on an open ball $V\subset \Sigma$ about $\sigma$. Since $f\in \frU$ and $\calA$ is open, we can get a neighborhood $V_0$ of $\sigma$, such that $\sigma\in V_0\subset V$ and $f|_{V_0}$ is a solution of $\calA$. In other words, $f|_{V_0}$ is admissible. Denote, $g_0 = \frD(f|_{V_0})$.
	
	Since $j^{\alpha+1}_f(\sigma)\in\calR_\alpha$, we have $j^{\alpha}_{g_0}(\sigma) = j^{\alpha}_{\frD(f)}(\sigma)  = 0$. Hence, for any given $\epsilon > 0$, there exists a neighborhood $W\subset V_0$ of $\sigma$ such that $|g_0|_\alpha < \epsilon$ on $W$. We can get some $g_\epsilon$ on $V_0$ so that,
	\begin{itemize}
		\item $g_\epsilon = -g_0$ on some neighborhood $W\subset V_0$ of $\sigma$, and 
		\item $g_\epsilon$ is $\epsilon$-small in $C^\alpha$-norm, i.e, $|g_\epsilon|_\alpha < \epsilon$ on $V_0$.
	\end{itemize}
	
	Now let us apply \autoref{thm:solution} for the domain $V_0$. Since $y_0 := f|_{V_0}$ is admissible, we have that $\frD_{y_0}$ admits a local inverse. In particular, there exists some $\epsilon,\delta > 0$ such that for any $|g|_{\alpha} < \epsilon$ we have unique $y$ such that $\frD(y)=\frD(y_0) + g$ and $|y-y_0|_{\alpha+1} < \delta$. Here we require that $\alpha$ to be sufficiently large. Now, in particular, for this $\epsilon = \epsilon(y_0, \alpha)$, we can get $W$ and $g_\epsilon$ as above. Since, $$|t g_\epsilon|_\alpha < t\epsilon < \epsilon, \quad\text{for $t\in [0,1]$,}$$ we have unique solutions $$f_t = \frD_{y_0}^{-1}(t g_\epsilon),$$
	over $V_0$, satisfying $|f_t - y_0|_{\alpha+1} < \epsilon$ for $t\in[0,1]$. Now, $$\frD(f_t) = \frD(y_0) + t g_\epsilon = \frD(f|_{V_0}) + t g_\epsilon= g_0 + tg_\epsilon.$$
	In particular, we have $\frD(f_0) = g_0$ and hence $f_0=u|_{V_0}$ from uniqueness. On the other hand, over $W$, $$\frD(f_1) = g_0 + g_\epsilon = g_0 - g_0 = 0.$$
	Thus $f_1$ is a solution $\frD(f_1) = 0$, over $W$. Furthermore, $f_t$ is admissible and $$j^\alpha_{\frD(f_t)}(\sigma) = j^\alpha_{g_0 + tg_\epsilon}(\sigma) = 0, \text{ as $g_\epsilon = -g_0$ on some $\Op(\sigma)$.}$$ Thus, $j^{\alpha+1}_{f_t}(\sigma)\in\calR_\alpha$ for all $t\in[0,1]$.
\end{proof}

We now have a (parametric) local $h$-principle for $\calR_\alpha$ (\cite[pg. 119]{gromovBook})
\begin{corr}\label{corr:localHPrinAlpha}
	For $\alpha$ large enough, the jet map $j^{\alpha+1} : \Sol \calR_\alpha \to \Gamma\calR_\alpha$ is a local weak homotopy equivalence.
\end{corr}

In order to prove the existence of a horizontal germ, i.e, a local solution of $\frD= 0$, we need to show that $\calR_\alpha \ne \emptyset$ at some $\sigma$. One issue with \autoref{thm:existenceOfLocalSolution} is that we do not specify the higher jet order $\alpha$ that is crucial in order to get a local solution. We now show that, in fact, we can get a lift to any arbitrary higher jet from the first jet relation of isotropic horizontal maps. Recall that given any map $f$ satisfying $f^*\lambda^i = 0$ we have, taking derivatives, that $f^*d\lambda^i = 0$. That is, $\im df$ is $d\lambda^i$-isotropic. Now from \autoref{prop:horizontalInversion}, we have that every solution is automatically admissible. On the other hand, we have the relation $\calR\subset\calR_0\subset J^1(\Sigma,M)$ consisting of jets $(x,y, F:T_x\Sigma\to T_y M)$ such that, $F^*d\alpha^s = 0$ for $s=1,2$. That is, sections of $\calR$ are bundle maps $F:T\Sigma\to TM$, which is a formal isotropic $\calD$-horizontal immersion. Observe that, $$\Sol\calR=\Sol\calR_\alpha,\;\text{for any $\alpha\ge 0$.}$$
We have the following result.

\begin{lemma}\label{lemma:jetLifting}
	For any $\alpha \ge 1$, the jet projection map $p=p^{\alpha + 1}_1 : J^{\alpha + 1}(\Sigma,M) \to J^1(\Sigma,M)$ maps $\calR_\alpha|_{(x,y)}$ surjectively onto $\calR|_{(x,y)}$, for any $(x,y)\in\Sigma\times M$. Furthermore, the fiber of $p$ over any jet in $\calR|_{(x,y)}$ is contractible and consequently, the induced map $\Gamma\calR_\alpha \to \Gamma\calR$ is a local weak homotopy equivalence.
\end{lemma}

We defer the details of the proof to \autoref{sec:jetLifting}. Let us first get the local $h$-principle.

\begin{proof}[Proof of \autoref{thm:hPrinciple}]
	From \autoref{corr:localHPrinAlpha}, we have that for $\alpha$ sufficiently large, the jet map $j^{\alpha + 1} : \Sol\calR \to \Gamma\calR_\alpha$ is a local weak weak homotopy equivalence. On the other hand, by \autoref{lemma:jetLifting}, the jet projection $p^{\alpha + 1}_1 : \Gamma\calR_\alpha \to \Gamma\calR$ is a weak homotopy equivalence. Hence the composition $$j^1 = p^{\alpha+1}_1 \circ j^{\alpha+1}: \Sol \calR \to \Gamma \calR$$ is a local weak homotopy equivalence. In other words, $\calD$-horizontal immersions satisfy the (parametric) local $h$-principle.
\end{proof}

Next, we get the existence of germs of $\calD$-horizontal $2$-submanifolds.

\begin{proof}[Proof of \autoref{thm:existenceOfGerms}]
	Suppose $\calD= \ker\lambda^1\cap \ker\lambda^2$ for some local $1$-forms $\lambda^i$ around some $y\in M$. Pick some arbitrary $0\ne v\in \calD_x$ and set $u = Av$, where $A$ is the (local) automorphism. Then observe that, $$d\lambda^1(u,v) = d\lambda^1(Av, v) = d\lambda^2(v,v) = 0 \quad\text{and,}\quad d\lambda^2(u,v) = d\lambda^1(u, Av) = d\lambda^1(u,u) = 0.$$
	In other words, $\Span \langle u, v\rangle \subset \calD_x$ is $\Omega$-isotropic. Now, consider the jet $\sigma = (0, y, F: T_0\bbD^2 \to T_y M) \in J^1(\bbD^2, M)$, given by, $$F(\partial_x) = u, \; F(\partial_y) = v.$$
	We clearly have $\sigma\in\calR|_{(0,y)}$ by construction. But then an application of the the local $h$-principle gives us that there exists a $\calD$-horizontal immersion $f: \Op(0) \to M$. Since $f$ is an immersion, it is a local diffeomorphism and thus we have a (germ of a) $\calD$-horizontal submanifold of dimension $2$.
\end{proof} 

\subsection{Proof of \autoref{lemma:jetLifting}}
\label{sec:jetLifting}
In this section, we discuss the proof of \autoref{lemma:jetLifting}. Instead of proving it only for fat distributions of rank $4$ on $6$-dimensional manifold, we consider an arbitrary fat distribution of corank $p$ on manifolds of dimension $N$. Since we are only considering jets of maps, let us consider $\Sigma = \bbR^2$ and $M=\bbR^N$, with fixed coordinates $\{x^1,x^2\}$ on $\Sigma$ and $\{y^1,\ldots,y^N\}$ on $M$. Suppose $\calD\subset TM$ is a corank $p$ fat distribution, given as the common kernel of $1$-forms $\lambda^1,\ldots,\lambda^p$ where we have $\lambda^s = \lambda^s_i dy^i$. For any $f:\Sigma\to M$ we have the operator $$\frD : f \mapsto \big(f^*\lambda^s\big)_{s=1,\ldots,p}$$
We need to understand the relation $\calR_\alpha\subset J^{\alpha+1}(\Sigma,M)$, which consists of jets $j^{\alpha+1}_f(x)$, where $f:\Op(x)\to M$ satisfies $j^\alpha_{\frD(f)}(x) = 0$. We also have the first jet relation $\calR\subset J^1(\Sigma,M)$ consisting of jets $(x,y,F:T_x\Sigma\to T_y M)$, where $F$ is an injective map taking $T_\sigma\Sigma$ to a common isotropic subspace of $\calD_y$. We prove the following stronger version of \autoref{lemma:jetLifting}.
\begin{lemma}\label{lemma:jetLiftingGeneral}
	The jet projection map $p = p^{\alpha+1}_1 : J^{\alpha+1}(\Sigma,M)\to J^1(\Sigma,M)$ maps $\calR_\alpha|_{(x,y)}$ surjectively onto $\calR|_{(x,y)}$. Furthermore, the fiber of $p$ over any fixed jet is contractible and consequently $p:\Gamma\calR_\alpha\to\Gamma\calR$ is a weak homotopy equivalence.
\end{lemma}

First we need to understand the equation $j^\alpha_{\frD(f)}(x)$ in terms of jets. We write down, $f^*\lambda^s = \eta^s_a dx^a$, where we have the functions $$\eta^s_a = f^*\lambda^s(\partial_a) = (\lambda^s_i\circ f) \partial_a f^i, \qquad 1\le s\le p, 1\le a \le 2.$$
Since $\frD(f) = \big(f^*\lambda^1,\ldots,f^*\lambda^p\big)$, we have that the jet $j^\alpha_{\frD(f)}(x)$ is nothing but $j^\alpha_{\eta^s_a}(x)$. We need to find out its higher derivatives. Let us fix our conventions first.\\

\paragraph{\bfseries Convention for Multi-Indices :} By a multi-index of order $r$ on the coordinates $\{x^1,x^2\}$ we will consider an \emph{ordered} tuple $I = (i_1 \le \ldots \le i_r)$ where $i_j\in\{1,2\}$. We denote, $$\partial_I = \partial_{i_1}\ldots\partial_{i_r} = \partial_{x^{i_1}} \ldots \partial_{x^{i_r}}.$$
A typical multi-index of order $\alpha$ looks like $\big(\underbrace{1,\ldots,1}_a,\underbrace{2,\ldots,2}_b\big)$, for some integers $a,b \ge 0$ satisfying $\alpha = a + b$. We denote $|I|$ as the order of a multi-index. A multi-index of order $1$ will be written without the parentheses. Let us denote by $\bbN_2^r$ the set of all multi-indices of order $r$ over the coordinates $\{x^1,x^2\}$.

For any subsequence $I^\prime\subset I$, we will denote $I - I^\prime$ as the multi-index obtained by taking the complimentary sequence. In particular for a given $I=(i_1,\ldots,i_r)$, we have $$I_j := I - i_j = \big(i_1,\ldots, i_{j-1},i_{j+1},\ldots, i_r\big).$$
Also, for given two multi-indices $I$ and $J$, by $I + J$ we will mean the multi-index obtained by first concatenating the two lists and then reordering in increasing order. In particular, observe that for any $i\in I$ we have, $(I - i)+ i = I$.\\

Now, for any multi-index $I=(i_1,\ldots,i_\alpha)$ of order $\alpha$, we have the higher order partial derivative,
\begin{align*}
	\partial_I \eta^s_a &= (\lambda^s_i \circ f) \partial_{I+a} f^i
	\\
	&\qquad + (\partial_\nu\lambda^s_i \circ f) \partial_I f^\nu \partial_a f^i + \sum_{l=1}^\alpha (\partial_\nu\lambda^s_i\circ f) \partial_{i_l} f^\nu \partial_{I-i_l + a} f^i
	\\
	&\qquad + \text{terms involving partial derivatives of $f$ of order $<\alpha$}
\end{align*}
Then, the condition $j^\alpha_{\frD(f)}(x) = 0$ is equivalent to the set of equations $$\big\{\partial_I \eta^s_a (x) = 0 \;\big|\; 1\le s\le p,\; 1\le a\le 2,\; |I| \le \alpha\}.$$
We now translate this system in the language of jets.\\

Consider an $\alpha+1$-jet $$\sigma = \big(x,\; y, \; P_i : \Sym^i T_x\Sigma \to T_y M,\; i=1,\ldots,\alpha+1\big)\in J^{\alpha+1}_{(x,y)}(\Sigma, M).$$
Note that the $i^\text{th}$ symmetric tensor $P_i$ is completely determined by the values $P^k_i(\partial_I)$, where $P^k_i$ are the components of $P_i$ and $\{\partial_I \;|\; I \in\bbN_2^i\}$ forms a basis of the symmetric space $\Sym^i T_x\Sigma$. Now, suppose the jet $\sigma$ is represented as $j^{\alpha+1}_f(x)$ for some $f:\Op(x)\to M$. Then for some $I=(i_1,\ldots,i_r)$ where $r\le\alpha$ and some $1\le a\le 2, \; 1\le s\le p$, the equation $\partial_I\eta^s_a(x) = 0$ gives us the following.
\begin{align*}
	\begin{aligned}
		\lambda^s_k(y) P_{r+1}^k(\partial_{I+a}) + \partial_\nu\lambda^s_k(y) P_1^k(\partial_a) P_{r}^\nu(\partial_I) + \sum_{l=1}^r \partial_\nu\lambda^s_k(y) P_1^\nu(\partial_{i_l}) P_{r}^k(\partial_{I-i_l+a}) + \text{terms involving $P_{ < r }$} = 0
	\end{aligned} \tag*{$\boxed{\partial_I \eta^s_a}$}
\end{align*}
Note that we have labeled these equations by $(\partial_I\eta ^s_a)$ as well. In particular, we have that the relation $\calR_\alpha$ consists of jets $\sigma\in J^{\alpha+1}(\Sigma,M)$ which satisfy the jet equations $(\partial_I\eta^s_a)$ for each $1\le s\le p, 1\le a\le 2$ and for each multi-index $|I|\le\alpha$.\\

Now, a jet $(x,y, P_1 : T_x\Sigma\to T_y M)\in J^1_{(x,y)}(\Sigma,M)$ belongs to $\calR_0|_{(x,y)}$ precisely when the equations
$$\lambda^s_k(y) P_1^k(\partial_a) = 0, \quad 1\le s\le p, 1\le a\le 2$$
hold. This is nothing but the condition that $P_1(T_x\Sigma)\subset \calD_y$. We also have the relation $\calR\subset\calR_0\subset J^1(\Sigma,M)$ consisting of jets $\sigma = (x,y,P_1)$ which further satisfy the isotropy condition. That is for each $v,w \in T_x \Sigma$ we have, $d\lambda^s(P_1(v), P_1(w)) = 0 \quad 1\le s\le p$. This isotropy condition can be expressed more explicitly as, $$\Big(\partial_\nu \lambda^s_\mu(y) - \partial_\mu\lambda^s_\nu(y)\Big) P_1^\nu(\partial_a) P_1^\mu(\partial_b) = 0,\quad\text{for each $1\le a,b\le 2$ and each $1\le s \le p$.}$$
For $\alpha\ge \beta$, we have the jet projection map $p = p^{\alpha+1}_{\beta+1}:J^{\alpha+1}(\Sigma, M) \to J^{\beta+1}(\Sigma,M)$ and clearly, $p$ maps $\calR_\alpha$ into $\calR_\beta$, since every defining equation for $\calR_\beta$ is also used to define $\calR_\alpha$. We now proceed with the proof of \autoref{lemma:jetLiftingGeneral}, which is essentially done via induction on $\alpha\ge 1$.\\

\paragraph{\bfseries Induction Statement $\calP(\alpha)$ :} For given $\alpha\ge 1$, the map $p=p^{\alpha+1}_1$ maps $\calR_\alpha$ into $\calR$. The system of equations involved in defining the relation $\calR_\alpha$ can be replaced by another system such that the following holds.
\begin{itemize}
	\item The solution space remains unchanged.
	\item The new system is a triangular one. That is, we can solve $P_1,P_2,\ldots$ etc successively.
	\item The highest order symmetric tensor $P_{\alpha+1}$ in the system appears as follows :
	\begin{equation}
		\begin{pmatrix}
			\lambda^1_1(y)&\ldots&\lambda^1_N(y)\\
			\vdots && \vdots \\
			\lambda^p_1(y) &\ldots &\lambda^p_N(y)
		\end{pmatrix}\begin{pmatrix}
			P_{\alpha+1}^1(\partial_J)\\ \vdots \\ P_{\alpha+1}^N(\partial_J)
		\end{pmatrix} = \text{a $p\times 1$-vector involving $P_{\le \alpha}$ terms}
		\tag*{$\boxed{\lambda_{\alpha+1}^J}$} \label{eqn:lambda:indHypo}
	\end{equation}
	for each multi-index $|J|=\alpha+1$. Note that we label these systems by $(\lambda_\alpha^J)$, since the coefficient matrix is the matrix of the $\bbR^p$-valued $1$-form $\lambda = (\lambda^1,\ldots,\lambda^p)$.
	\item The system is consistent and admits solutions. Furthermore, the solution space is contractible.
\end{itemize}

The induction relies on two facts : We will see that $\calP(1)$ holds because by assumption the $1$-forms $\lambda^1,\ldots,\lambda^p$ are independent at each point and this independence enables us to solve certain affine system. Whereas, for any $\alpha\ge 1$, to get $\calP(\alpha + 1)$ from $\calP(\alpha)$, we need to utilize the fact the distribution $\calD$ is in fact fat. Let us now proceed with the details.\\

\paragraph{\textbf{Induction Base Step $\calP(1)$ :}} We focus on the relation $\calR_1\subset J^2(\Sigma, M)$. Consider some jet, $$\tilde \sigma=(x,y,P_1:T_x\Sigma\to T_y M, P_2 : \Sym^2 T_x \Sigma\to T_y M)\in\calR|_{(x,y)}$$
For $1\le a < b\le 2$, i.e, for $a=1,b=2$, we have the equations defining $\calR_1$ as follows :
\begin{align*}
	\lambda^s_k(y)  P_2^k(\partial_{ab}) + \partial_\nu\lambda^s_k(y)P_1^k(\partial_b)P_1^\nu(\partial_a) = 0 \tag*{$\boxed{\partial_a\eta^s_b}$}
	\\
	\lambda^s_k(y) P_2^k(\partial_{ab}) + \partial_\nu \lambda^s_k(y) P_1^k(\partial_a)P_1^\nu(\partial_b) = 0 \tag*{$\boxed{\partial_b\eta^s_a}$}
\end{align*}
Since $\tilde\sigma\in\calR_1|_{(x,y)}$ must satisfy both $(\partial_a\eta_b^s)$ and $(\partial_b\eta_a^s)$ above, we perform $(\partial_a\eta^s_b) - (\partial_b\eta^s_a)$ and get
\begin{align*}
	0 = \Big(\partial_\nu\lambda^s_k(y) - \partial_k\lambda^s_\nu(y) \Big) P_1^\nu(\partial_a)P_1^k(\partial_b).
\end{align*}
But observe that the isotropy condition on $P_1$ is precisely $$0=d\lambda^s|_y\big(P_1(\partial_a),P_1(\partial_b)\big) = \Big(\partial_\nu\lambda^s_k(y) - \partial_k\lambda^s_\nu(y) \Big) P_1^\nu(\partial_a)P_1^k(\partial_b).$$
Hence we see that $p(\tilde \sigma) = (x,y, P_1)\in \calR|_{(x,y)}$, as $P_1$ satisfies the isotropy condition. Thus, $p$ maps $\calR_1$ into $\calR$.

Next, assume that $\sigma = (x,y, P_1:T_x\Sigma\to T_y M)\in\calR|_{(x,y)}$ is given. We need to look for a jet $\tilde{\sigma} = (x,y,P_1,P_2)\in\calR_1|_{(x,y)}$. That is, we need to find out $P_2$ satisfying the equations $\{\partial_a \eta^s_b \big| 1\le s\le p, 1\le a\le b\le 2\}$. Now for $a < b$, we see that the terms $P_2^k(\partial_{ab})$ appear linearly in both the systems $\{\partial_a\eta^s_b|1\le s\le p\}$ and $\{\partial_b\eta^s_a | 1\le s \le p\}$, with \emph{identical} coefficients. Thus we run into a question of consistency. But we have seen that the expression $(\partial_a\eta^s_b) - (\partial_b\eta^s_a)$ is identical to the isotropy condition for $P_1$. Since the jet $\sigma\in \calR$, we know that $P_1$ satisfies the isotropy condition. Thus we have that the equations $(\partial_a\eta^s_b) - (\partial_b\eta^s_a)$ is zero and so for each tuple $a\le b$, we can keep the equation labeled by $\partial_a\eta^s_b$ and remove the equation $\partial_b\eta_b^s$. We are left with the system
\begin{equation}
	\begin{pmatrix}
		\lambda^1_1(y)&\ldots&\lambda^1_N(y)\\
		\vdots && \vdots \\
		\lambda^p_1(y) &\ldots &\lambda^p_N(y)
	\end{pmatrix}\begin{pmatrix}
		P_{2}^1(\partial_{ab})\\ \vdots \\ P_{2}^N(\partial_{ab})
	\end{pmatrix} = \text{a $p\times 1$ vector involving $P_1$} \tag*{$\boxed{\lambda_2^{ab}}$} \label{eqn:lambda:1}
\end{equation}
Clearly the solution space remains unchanged.

Lastly, to show that $P_2$ can be solved, note that the above \emph{affine} system has full rank coefficient matrix, since the rows are nothing but the $1$-forms $\lambda^s$, which are given to be independent. Thus, the system admits a solution. We have proved that $p:\calR_1|_{(x,y)}\to \calR|_{(x,y)}$ is surjective with affine fibers.\\

\paragraph{\bfseries Induction Hypothesis $\calP(\alpha)$ :} Assume that for some $\alpha\ge 1$, the statement $\calP(\alpha)$ holds.\\

\paragraph{\bfseries Induction Step $\calP(\alpha) \Rightarrow \calP(\alpha+1)$ :} First observe that the jet map $p$ maps $\calR_{\alpha+1}$ into $\calR$; since we have already proved this for $\calR_1$ and the equations involved in $\calR_1$ are also present in $\calR_{\alpha+1}$. We prove that $p$ is surjective, with contractible fiber.

Fix a jet $\sigma=(x,y,P_1)\in\calR|_{(x,y)}$. Since the equation system defining $\calR_\alpha$ is included in $\calR_{\alpha+1}$, using the induction hypothesis $\calP(\alpha)$, we replace this (sub)system with the  triangular system, keeping the equations involving $P_{\alpha+2}$ untouched. Next, we solve for the tensors $P_r$ with $r \le \alpha$ from this transformed system. Note that, we could also solve for the tensors $P_{\alpha + 1}$ using the induction hypothesis, but we defer this for later. As we will see that while solving for the tensor $P_{\alpha+2}$ we will run into some consistency problem, which will introduce new sets of equations for $P_{\alpha+1}$. The following ladder like diagram gives a schematic representation of this step :
\begin{displaymath}
	\tikzset{%
		symbol/.style={
			draw=none,
			every to/.append style={
				edge node={node [sloped, allow upside down, auto=false]{$#1$}}
			},
		},
	}
	\begin{tikzcd}[row sep=.8cm]
		J^{\alpha+2}(\Sigma,M) \arrow{r}{p^{\alpha+2}_{\alpha+1}} & J^{\alpha+1}(\Sigma,M) \arrow{r}{p^{\alpha+1}_\alpha} & J^{\alpha}(\Sigma,M) \arrow{r} &\cdots \arrow{r} & J^2(\Sigma,M) \arrow{r}{p^2_1} & J^1(\Sigma,M)\\
		\calR_{\alpha+1} \arrow{r} \arrow[rdd, twoheadrightarrow, sloped, "p^{\alpha+2}_{\alpha+1}"] \arrow[symbol=\subset]{u} &\calR_{\alpha} \arrow[twoheadrightarrow]{rdd} \arrow{r} \arrow[symbol=\subset]{u} & \calR_{\alpha-1} \arrow{r} \arrow[symbol=\subset]{u} & \cdots \arrow{r} & \calR_1 \arrow{r} \arrow[twoheadrightarrow]{rdd} \arrow[symbol=\subset]{u} & \calR_0 \arrow[symbol=\subset]{u} \\
		\\
		&\calS_{\alpha} \arrow[uu, hookrightarrow, sloped, "\text{\tiny additional}", "\text{\tiny affine eqns}"'] \arrow[uul, dashed, bend left=40, sloped,  "\text{\tiny lift using}", "\text{\tiny independence of forms}"'] & S_{\alpha-1} \arrow[hookrightarrow]{uu} \arrow[l, dashed, swap, "\text{\tiny lift using}", "\text{\tiny fatness of $\calD$}"'] &&& S_0 = \calR \arrow[uu, hookrightarrow, sloped, "\text{\tiny isotropy}", "\text{\tiny (quadratic)}"'] \arrow[lll, dashed, swap, "\text{lift to $\calR_\alpha$ using induction hypothesis}", "\text{and then project to $\calS_{\alpha-1}$}"'] 
	\end{tikzcd} \label{eqn:diagram} \tag{$\dagger$}
\end{displaymath}
where, we have denoted $S_\beta = p^{\beta+2}_{\beta+1}(\calR_{\beta+1})$ as the image. So, using the induction hypothesis, we first get a lift of the jet $\sigma$ to $\calS_{\alpha-1}= p^{\alpha+1}_\alpha(\calR_\alpha)$. We now need to identify $\calS_\alpha = p^{\alpha+2}_{\alpha+1}(\calR_{\alpha+1})$, which is defined by the new set of equations coming from the consistency.\\

Let us first fix some more notations for the multi-indices. Recall that $\bbN_2^r$ is the set of all multi-indices $I$ on $\{x^1,x^2\}$, with order $|I|=r$. Now for some $J\in \bbN_2^{r+1}$ we denote $$\frI(J) = \Big\{(I,a) \;\Big|\; I\in \bbN_2^r, \; 1\le a\le 2, J=I+a\Big\}$$
Then observe that for any $(I,a)\in\frI(J)$ and some $1\le s\le p$, the equation labeled by $\partial_{I}\eta^s_a$ involves the terms $P_{r+1}^k(\partial_J)$. Explicitly, if we consider $J=(a_0, a_1,\ldots, a_r)$ with $1\le a_0\le \ldots \le a_r\le 2$, then we get $$\frI(J) = \bigcup_{i=0}^r \Big\{ (I_i, a_i) \;\Big|\; I_i := I - a_i = \big(a_0,\ldots,\hat{a_i},\ldots,a_r\big) \Big\}$$
Since we may have repetitions in the $a_i$'s, we see that $|\frI(J)| \le |J| = r+1$. In fact, since we have only two indices $\{x^1,x^2\}$, we can see that $|\frI(J)| \le 2$.\\

We make the following observation. Fix some $J=(a_0,\ldots,a_{\alpha+1})\in\bbN_2^{\alpha+2}$ with $1\le a_0\le \ldots\le a_{\alpha+1}\le 2$. Consider $(I_0,a_0)$ and $(I_j,a_j)$ in $\frI(J)$ for some $1\le j\le r + 1$ fixed. Then for each $1\le s\le p$, we have the equations $\partial_{I_0}\eta^s_{a_0}$ and $\partial_{I_j}\eta^s_{a_j}$ as follows :
\begin{align*}
	\begin{aligned}
		&\lambda^s_k(y) P_{\alpha+2}^k(\partial_J) + \partial_\nu \lambda^s_k(y) P_1^k(\partial_{a_0})P_{\alpha+1}^\nu(\partial_{I_0}) + \sum_{i=1}^{\alpha+1} \partial_\nu\lambda^s_k(y) P_1^\nu(\partial_{a_i}) P_{\alpha+1}^k(\partial_{J - a_i})
		\\
		&\hspace{5cm} + \text{terms involving $P_{\le \alpha}$} = 0
	\end{aligned} \tag*{$\boxed{\partial_{I_0}\eta^s_{a_0}}$} \label{eqn:jet:keep}
	\\
	\begin{aligned}
		&\lambda^s_k(y) P_{\alpha+2}^k(\partial_J) + \partial_\nu\lambda^s_k(y) P_1^k(\partial_{a_j})P_{\alpha+1}^\nu(\partial_{I_j}) + \sum_{\substack{i=0\\i\ne j}}^{\alpha+1} \partial_\nu\lambda^s_k(y) P_1^\nu(\partial_{a_i}) P_{\alpha+1}^k(\partial_{J-a_i})
		\\
		&\hspace{5cm} + \text{terms involving $P_{ \le \alpha}$} = 0
	\end{aligned} \tag*{$\boxed{\partial_{I_j}\eta^s_{a_j}}$} \label{eqn:jet:delete}
\end{align*}

If $a_0 = a_j$, then $I_0 = I_j$ and hence the equations $\partial_{I_0}\eta^s_{a_0}$ and $\partial_{I_j}\eta^s_{a_j}$ becomes identical. So without a loss of generality, we assume that $a_0 < a_j$. If no such $j$ exists, then we must have $J = (\underbrace{a_0,\ldots,a_0}_{\text{$\alpha + 2$-times}})$ which gives a single equation $\partial_{I_0}\eta^s_{a_0}$ to consider. Note that the only way $a_0 < a_j$ can hold is for some $j$ is $a_0 = 1, a_j = 2$.

Now, observe that the two systems $\{\partial_{I_0}\eta^s_{a_0}, 1\le s\le p\}$ and $\{\partial_{I_j}\eta^s_{a_j}, 1\le s\le p\}$ both look like
\begin{equation}
	\begin{pmatrix}
		\lambda^1_1(y)&\ldots&\lambda^1_N(y)\\
		\vdots && \vdots \\
		\lambda^p_1(y) &\ldots &\lambda^p_N(y)
	\end{pmatrix}\begin{pmatrix}
		P_{\alpha+2}^1(\partial_J)\\ \vdots \\ P_{\alpha + 2}^N(\partial_J)
	\end{pmatrix} = \text{a $p\times 1$ vector involving $P_{\le\alpha+1}$} \tag*{$\boxed{\lambda_{\alpha+2}^J}$} \label{eqn:lambda:alpha1}
\end{equation}
And thus we run into the question of consistency from these two systems. To address this issue, we first perform $(\partial_{I_0}\eta^s_{a_0})-(\partial_{I_j}\eta^s_{a_j})$ and get
\begin{equation}
	\begin{aligned}
		&\Big(\partial_\nu\lambda^s_k(y) - \partial_k\lambda^s_\nu(y)\Big) \Big(P_1^k(\partial_{a_0})P_{\alpha+1}^\nu (\partial_{I_0}) - P_1^k(\partial_{a_j}) P_{\alpha+1}^\nu(\partial_{I_j})\Big)\\
		&\hspace{4cm} + \text{terms involving $P_{\le \alpha}$} = 0
	\end{aligned}\tag*{$\boxed{(\partial_{I_0}\eta^s_{a_0}) - (\partial_{I_j}\eta^s_{a_j})}$} \label{eqn:jet:diff}
\end{equation}

Note that the difference does not involve any $P_{\alpha + 2}$ terms at all. We keep the equations labeled by (\hyperref[eqn:jet:keep]{$\partial_{I_0}\eta^s_{a_0}$}) and for each $j=1,\ldots, \alpha+1$, we replace the equation (\hyperref[eqn:jet:delete]{$\partial_{I_j}\eta^s_{a_j}$}) by the equation (\hyperref[eqn:jet:diff]{$(\partial_{I_0}\eta^s_{a_0}) - (\partial_{I_j}\eta^s_{a_j})$}); provided $a_0 < a_j$. Clearly this does not change the solution space, but introduces \emph{new} set of equations involving $P_{\le \alpha + 1}$. The system is still affine. These added set of equations, together with the original system involving $P_{\alpha+1}$, now define $\calS_{\alpha} = p^{\alpha+2}_{\alpha+1}(\calR_{\alpha+1})$ (see diagram (\ref{eqn:diagram})). Note that, at this point we are left with exactly one system that involves $P_{\alpha+2}^k(\partial_J)$, which looks like (\hyperref[eqn:lambda:alpha1]{$\lambda^J_{\alpha+2}$}) as above.

Let us now fix the dictionary order $\prec$ on $\bbN_2^{\alpha+1}$, induced by the obvious ordering of the coordinate indices. Then we have $I_j \prec I_0$, since the first position they differ must be larger than $a_0$. We rewrite the equation (\hyperref[eqn:jet:diff]{$(\partial_{I_0}\eta^s_{a_0}) - (\partial_{I_j}\eta^s_{a_j})$}) as
\begin{equation}
	\begin{aligned}
		\Big(\partial_\nu\lambda^s_k(y) - \partial_k\lambda^s_\nu(y)\Big) P_1^k(\partial_{a_0})P_{\alpha+1}^\nu (\partial_{I_0}) &= \Big(\partial_\nu\lambda^s_k(y) - \partial_k\lambda^s_\nu(y)\Big) P_1^k(\partial_{a_j}) P_{\alpha+1}^\nu(\partial_{I_j})\\
		&\qquad + \text{terms with $P_{\le \alpha}$}
	\end{aligned} \tag*{$\boxed{(\partial_{I_0}\eta^s_{a_0}) - (\partial_{I_j}\eta^s_{a_j})}$} \label{eqn:jet:diff1}
\end{equation}
and add these to the set of equations that is used to solve for $P^k_{\alpha+1}(I_0)$, which, from from the induction hypothesis, is given as the system (\hyperref[eqn:lambda:indHypo]{$\lambda^{I_0}_{\alpha+1}$}).

Observe that we are adding at most one new system of $p$-many equations for each tensor $P_{\alpha + 1}^k(\partial_I)$, since we have $|\frI(I)|\le 2$. In fact, it is clear that for any fixed $|I|=\alpha+1$, the system that we are adding looks like
\begin{align*}
	\Big(\partial_\nu\lambda^s_k(y) - \partial_k\lambda^s_\nu(y)\Big) P_1^k(\partial_{1})P_{\alpha+1}^\nu (\partial_{I}) = \text{known terms with $P_{\le \alpha}$ and $P_{\alpha+1}(\partial_{I^\prime})$ with $I^\prime \prec I$,} \quad \text{for $s= 1,\ldots,p$.}
\end{align*}
Furthermore, we are not adding \emph{any} equation when $I = \hat I :=  (\underbrace{1,\ldots,1}_{\alpha+1})$, which is the least element in the ordering $\prec$.\\

We claim that we are able to solve the $P_{\alpha+1}$-terms in a triangular fashion, ordered by $\prec$. To see this, first observe that the only equation system that involves $P_{\alpha+1}(\partial_{\hat I})$ looks like (\hyperref[eqn:lambda:indHypo]{$\lambda^{\hat I}_{\alpha+1}$}). But this affine system admits solutions, since the coefficient matrix has full rank. Now, inductively assume that for some multi-index $I$, with $|I| = \alpha+1$ and $\hat I \prec I$, we have solved the tensor $P_{\alpha + 1}(\partial_{I^\prime})$ for any $I^\prime\prec I$. For this $I$, we have added the following set of equations,
\begin{align*}
	\Big(\partial_\nu\lambda^s_k(y) - \partial_k\lambda^s_\nu(y)\Big) P_1^k(\partial_{a})P_{\alpha+1}^\nu (\partial_{I}) = \text{known terms with $P_{\le \alpha}$ and $P_{\alpha+1}(\partial_{I^\prime})$ with $I^\prime \prec I$,}
\end{align*}
for each $1\le s\le p$, to the system (\hyperref[eqn:lambda:indHypo]{$\lambda^I_{\alpha+1}$}). That is, we have an \emph{affine} system in $P_{\alpha+1}(\partial_I)$ given as  follows :
\begin{equation}
	\begin{pmatrix}
		\lambda^1_1(y) &\ldots &\lambda^1_N(y) \\
		\vdots && \vdots \\
		\lambda^p_1(y) &\ldots &\lambda^p_N(y) \\
		\\
		\big(\partial_\nu\lambda^1_1(y) - \partial_1\lambda^1_\nu (y)\big) P_1^\nu(\partial_1) &\ldots & \big(\partial_\nu\lambda^1_N(y) - \partial_N\lambda^1_\nu (y)\big) P_1^\nu(\partial_1)\\
		\vdots && \vdots \\
		\big(\partial_\nu\lambda^p_1(y) - \partial_1\lambda^p_\nu (y)\big) P_1^\nu(\partial_1) &\ldots & \big(\partial_\nu\lambda^p_N(y) - \partial_N\lambda^p_\nu (y)\big) P_1^\nu(\partial_1)
	\end{pmatrix}\begin{pmatrix}
		P_{\alpha+1}^1(\partial_I) \\ \vdots \\ P_{\alpha+1}^N(\partial_I)
	\end{pmatrix} =  \parbox{2.5cm}{\centering \small A $2p\times 1$-vector of known terms involving $P_{\le \alpha}$, and $P_{\alpha+1}(\partial_{I^\prime})$ with $I^\prime\prec I$}
	\tag*{$\boxed{d\lambda_{\alpha+1}^I}$} \label{eqn:lamda:dlambda}
\end{equation}
Note that the rows of the $(2p\times N)$-sized coefficient matrix above are, respectively, the $1$-forms on $T_y M$, $$\lambda^1\big|_y, \ldots, \lambda^p\big|_y,\quad  -\iota_{P(\partial_1)}d\lambda^1\big|_y, \ldots,-\iota_{P(\partial_1)}d\lambda^p\big|_y,$$
written with respect to the basis $\{dy^1,\ldots,dy^N\}$. Hence the matrix is full rank precisely when we have the wedge $$\big(\lambda^1 \wedge\ldots \wedge \lambda^p\big) \wedge \big(\iota_{P_1(\partial_1)}d\lambda^1\big) \wedge\ldots\wedge \big(\iota_{P_1(\partial_1)}d\lambda^p\big)$$
is non-zero. But since $\calD$ is taken to be fat, for any vector $0\ne v\in\calD_y$, and so in particular for $v = P_1(\partial_1)$, we have that $$\lambda^1\wedge\ldots\wedge\lambda^p\wedge\iota_vd\lambda^1\wedge\ldots\wedge\iota_vd\lambda^p \ne 0,$$
which follows from \autoref{rmk:fatAndSymplectic}. Thus the coefficient matrix in the system (\hyperref[eqn:lamda:dlambda]{$d\lambda^I_{\alpha+2}$}) above indeed has full rank. We can then inductively solve the tensor $P_{\alpha+1}$ completely. At this point, we have a lift of the jet $\sigma$ to $\calS_\alpha$. 

Lastly, for each multi-index $J$ of order $\alpha+2$, we can easily solve $P_{\alpha+2}(\partial_J)$ from the affine systems (\hyperref[eqn:lambda:alpha1]{$\lambda^J_{\alpha+2}$}), which has full rank coefficient matrix. Thus we have obtained a jet $\tilde{\sigma}\in\calR_{\alpha+1}$ so that $p(\tilde{\sigma})=\sigma$. Clearly the solution space $p^{-1}(\sigma)$ is contractible, since at each stage we have solved affine system of equations, in a triangular fashion. Thus we have proved that $\calP(\alpha+1)$ is true.\\

This concludes the induction. Furthermore, it is clear from the algorithmic approach above that we can get lift of arbitrary sections of $\calR$ to $\calR_\alpha$, along the map $p$, using local triviality arguments. Then, from the sheaf theoretic argument presented in \cite[pg. 76-78]{gromovBook}, we have that $p : \Gamma\calR_\alpha \to \Gamma\calR$ is a weak homotopy equivalence. This completes the proof of \autoref{lemma:jetLiftingGeneral}.

\section{Appendix : Hamilton's Implicit Function Theorem}
\label{sec:hamiltonIFT}
Nash's Implicit Function Theorem \cite{nashSmoothIsometric} in the context of $C^\infty$-isometric immersions has been generalized by several authors. Here we recall Hamilton's formalism of an infinite dimensional implicit function theorem that works for smooth differential operators between Fr\'echet spaces. This theorem is used crucially in order to get the local h-principle (\autoref{thm:openInversion}) of horizontal maps into corank $2$ fat distributions which admit Reeb directions. To begin with, we discuss the basic notions of tame spaces and tame operators from the exposition by Hamilton (\cite{hamiltonNashMoser}).

\begin{defn}{\normalfont \cite[pg. 67]{hamiltonNashMoser}}\label{defn:frechetSpace}\index{Fr\'echet space}
	A \emph{Fr\'echet space} is a complete, Hausdorff, metrizable, locally convex topological vector space.
\end{defn}

In particular the topology of a Fr\'echet space $F$ is given by a countable collection of semi-norms $\{|\cdot|_n\}$, such that a sequence $f_j\to f$ if and only if $|f_j - f|_n\to 0$ for all $n$, as $j\to \infty$. A choice of this collection of norms is called a grading on the space and we say $(F,\{|\cdot|_n\})$ is a \emph{graded} Fr\'echet space.
\begin{example}\label{exmp:frechetSpaces}
	Many naturally occurring spaces are in fact Fr\'echet spaces.
	\begin{enumerate}
		\item Every Banach space $(X,|\cdot|_X)$ is a Fr\'echet space. It may also be graded if we set $|\cdot|_n=|\cdot|_X$ for all $n$ (\cite[pg. 68]{hamiltonNashMoser}).
		
		\item Given a compact manifold $X$, possibly with boundary, the function space $C^\infty(X)$ is a graded Fr\'echet space. More generally, given any vector bundle $E\to X$, the space of sections $\Gamma (E)$ is also a graded Fr\'echet space. The $C^k$-norms on the sections give a possible grading (\cite[pg. 68]{hamiltonNashMoser}).
		
		\item Given a Banach space $(X,|\cdot|_X)$, denote by $\Sigma (X)$ the space of exponentially decreasing sequences of $X$, which consists of sequences $\{x_k\}$ of elements of $X$, such that, $$|\{x_k\}|_n = \sum_{k=0}^\infty e^{nk} |x_k|_X  < \infty,\quad \forall n\ge 0.$$
		Then $\Sigma(X)$ is a graded Fr\'echet space with the norms defined above (\cite[pg. 134]{hamiltonNashMoser}).
	\end{enumerate}
\end{example}
\begin{defn}{\normalfont \cite[pg. 135]{hamiltonNashMoser}}\label{defn:tameLinearMap}\index{tame linear map}
	A linear map $L:F\to G$ between Fr\'echet spaces $F,G$ is said to satisfy \emph{tame estimates} of degree $r$ and base $b$ if there exists a constant  $c=c(n)$ such that, $$|Lf|_n \le C|f|_{n+r}, \quad \forall n\ge b,\quad \forall f\in F.$$
	$L$ is said to be tame if it satisfies the tame estimates for some $n$ and $r$.
\end{defn}
\begin{example} \label{exmp:tameOperators} 
	We have that a large class of operators are in fact tame.
	\begin{enumerate}
		\item \label{exmp:tameOperators:1} A linear partial differential operator $L:C^\infty(X)\to C^\infty(X)$ of order $r$ satisfies the tame estimate $|Lu|_n \le |u|_{n+r}$ for all $n\ge 0$ and hence $L$ is tame of degree $r$ (\cite[pg. 135]{hamiltonNashMoser}).
		
		\item \label{exmp:tameOperators:2} Inverses of elliptic, parabolic, hyperbolic and sub-elliptic operators are tame maps (\cite[pg. 67]{hamiltonNashMoser}). In particular, the solution of an elliptic boundary value problem is tame (\cite[pg. 161]{hamiltonNashMoser}).
		
		\item \label{exmp:tameOperators:3} Composition of two tame maps is again tame (\cite[pg. 136]{hamiltonNashMoser}).
	\end{enumerate}
\end{example}

\begin{defn}{\normalfont \cite[pg. 136]{hamiltonNashMoser}} \label{defn:tameDirectSummand}
	Given graded Fr\'echet spaces $F,G$, we say $F$ is a \emph{tame direct summand} of $G$ if there are tame linear maps $L:F\to G$ and $M:G\to F$ such that the composition $ML:F\to F$ is the identity.
\end{defn}


\begin{defn}{\normalfont \cite[pg. 136]{hamiltonNashMoser}}\label{defn:tameFrechetSpace}
	A Fr\'echet space $F$ is said to be \emph{tame} if $F$ is a tame direct summand of $\Sigma(X)$, for some Banach space $X$.
\end{defn}

\begin{example}
	\label{exmp:sectionSpaceIsTame}
	Given a compact manifold $X$, possibly with boundary, and a vector bundle $E\to X$, the section space $\Gamma(E)$ is a \emph{tame} Fr\'echet space (\cite[pg. 139]{hamiltonNashMoser}).
\end{example}

\begin{defn}\cite[pg. 143]{hamiltonNashMoser}\label{defn:tameSmoothMap}\index{tame smooth map}
	A map $P:U\subset F\to G$ between Fr\'echet spaces $F$ and $G$, defined over some open set $U\subset F$, is said to be a \emph{smooth tame map} if $P$ is smooth and all the derivatives $D^kP$ are tame linear maps.
\end{defn}

We now state the inverse function theorem.

\begin{theorem}{\normalfont \cite[pg. 171]{hamiltonNashMoser}}\label{thm:hamiltonNashMoser}
	Consider tame Fr\'echet spaces $F,G$ and a tame smooth map $P:U \to G$, where $U\subset F$ is open. Suppose for the derivative $DP(f)$ at $f\in U$, the equation $DP(f)h = k$ admits unique solution $h=VP(f)k$ for each $k\in G$. Furthermore, assume that $VP:U\times G\to F$ is a smooth tame map. Then $P$ is locally invertible and each local inverse $P^{-1}$ is smooth tame.
\end{theorem}

\begin{remark}\label{rmk:hamiltonIFTBannachIFT}
	Unlike the inverse function theorem for Banach spaces, one needs to have that the derivative $DP$ is invertible on an \emph{open} set $U\subset F$.
\end{remark}

\section*{Acknowledgment}
The author would like to thank Mahuya Datta, Adi Adimurthi and Partha Sarathi Chakraborty for fruitful and enlightening discussions. The author would also like to thank the anonymous referee for many valuable comments and suggestions.

\bibliographystyle{alphaurl}
\bibliography{references}

\begin{thebibliography}{Ham82}

\bibitem[AFL17]{forstnericHoloLegendrianCurves}
Antonio Alarc\'{o}n, Franc Forstneri\v{c}, and Francisco~J. L\'{o}pez.
\newblock Holomorphic {L}egendrian curves.
\newblock {\em Compositio Mathematica}, 153(9):1945--1986, 2017.
\newblock \href {https://doi.org/10.1112/S0010437X1700731X}
  {\path{doi:10.1112/S0010437X1700731X}}.

\bibitem[BH05]{bandeContactPair}
Gianluca Bande and Amine Hadjar.
\newblock Contact pairs.
\newblock {\em The Tohoku Mathematical Journal. Second Series}, 57(2):247--260,
  2005.

\bibitem[CdS01]{daSilvaBook}
Ana Cannas~da Silva.
\newblock {\em Lectures on symplectic geometry}, volume 1764 of {\em Lecture
  Notes in Mathematics}.
\newblock Springer-Verlag, Berlin, 2001.
\newblock \href {https://doi.org/10.1007/978-3-540-45330-7}
  {\path{doi:10.1007/978-3-540-45330-7}}.

\bibitem[CFS05]{nilpoten6DimLieAlgebra}
Sergio Console, Anna Fino, and Evangelia Samiou.
\newblock The moduli space of six-dimensional two-step nilpotent {L}ie
  algebras.
\newblock {\em Annals of Global Analysis and Geometry}, 27(1):17--32, 2005.
\newblock \href {https://doi.org/10.1007/s10455-005-2569-2}
  {\path{doi:10.1007/s10455-005-2569-2}}.

\bibitem[Cho39]{chowBracketGenerating}
Wei-Liang Chow.
\newblock \"{U}ber {S}ysteme von linearen partiellen {D}ifferentialgleichungen
  erster {O}rdnung.
\newblock {\em Mathematische Annalen}, 117:98--105, 1939.
\newblock \href {https://doi.org/10.1007/BF01450011}
  {\path{doi:10.1007/BF01450011}}.

\bibitem[D'A94]{dAmbraSubbundle}
Giuseppina D'Ambra.
\newblock Induced subbundles and {N}ash's implicit function theorem.
\newblock {\em Differential Geometry and its Applications}, 4(1):91--105, 1994.
\newblock \href {https://doi.org/10.1016/0926-2245(94)00003-4}
  {\path{doi:10.1016/0926-2245(94)00003-4}}.

\bibitem[Duc84]{duchampLegendre}
Tom Duchamp.
\newblock The classification of {L}egendre immersions.
\newblock 1984.

\bibitem[EM02]{eliashbergBook}
Y.~Eliashberg and N.~Mishachev.
\newblock {\em Introduction to the {$h$}-principle}, volume~48 of {\em Graduate
  Studies in Mathematics}.
\newblock American Mathematical Society, Providence, RI, 2002.
\newblock \href {https://doi.org/10.1090/gsm/048} {\path{doi:10.1090/gsm/048}}.

\bibitem[FL18]{forstnericLegendrianCurves}
Franc Forstneri\v{c} and Finnur L\'{a}russon.
\newblock The {O}ka principle for holomorphic {L}egendrian curves in
  {$\Bbb{C}^{2n+1}$}.
\newblock {\em Mathematische Zeitschrift}, 288(1-2):643--663, 2018.
\newblock \href {https://doi.org/10.1007/s00209-017-1904-1}
  {\path{doi:10.1007/s00209-017-1904-1}}.

\bibitem[Ge92]{geCharacteristicClass}
Zhong Ge.
\newblock Betti numbers, characteristic classes and sub-{R}iemannian geometry.
\newblock {\em Illinois Journal of Mathematics}, 36(3):372--403, 1992.
\newblock \href {https://doi.org/10.1215/ijm/1255987416}
  {\path{doi:10.1215/ijm/1255987416}}.

\bibitem[Ge93]{geHorizontalCC}
Zhong Ge.
\newblock Horizontal path spaces and {C}arnot-{C}arath\'{e}odory metrics.
\newblock {\em Pacific Journal of Mathematics}, 161(2):255--286, 1993.
\newblock \href {https://doi.org/10.2140/PJM.1993.161.255}
  {\path{doi:10.2140/PJM.1993.161.255}}.

\bibitem[Gei08]{geigesBook}
Hansj\"org Geiges.
\newblock {\em An introduction to contact topology}, volume 109 of {\em
  Cambridge Studies in Advanced Mathematics}.
\newblock Cambridge University Press, Cambridge, 2008.
\newblock \href {https://doi.org/10.1017/CBO9780511611438}
  {\path{doi:10.1017/CBO9780511611438}}.

\bibitem[Gro86]{gromovBook}
Mikhael Gromov.
\newblock {\em Partial differential relations}, volume~9 of {\em Ergebnisse der
  Mathematik und ihrer Grenzgebiete (3) [Results in Mathematics and Related
  Areas (3)]}.
\newblock Springer-Verlag, Berlin, 1986.
\newblock \href {https://doi.org/10.1007/978-3-662-02267-2}
  {\path{doi:10.1007/978-3-662-02267-2}}.

\bibitem[Gro96]{gromovCCMetric}
Mikhael Gromov.
\newblock Carnot-{C}arath\'{e}odory spaces seen from within.
\newblock In {\em Sub-{R}iemannian geometry}, volume 144 of {\em Progr. Math.},
  pages 79--323. Birkh\"{a}user, Basel, 1996.
\newblock URL:
  \url{https://www.ihes.fr/~gromov/wp-content/uploads/2018/08/carnot_caratheodory.pdf}.

\bibitem[Ham82]{hamiltonNashMoser}
Richard~S. Hamilton.
\newblock The inverse function theorem of {N}ash and {M}oser.
\newblock {\em American Mathematical Society. Bulletin. New Series},
  7(1):65--222, 1982.
\newblock \href {https://doi.org/10.1090/S0273-0979-1982-15004-2}
  {\path{doi:10.1090/S0273-0979-1982-15004-2}}.

\bibitem[Mon93]{montGeneric}
Richard Montgomery.
\newblock Generic distributions and {L}ie algebras of vector fields.
\newblock {\em Journal of Differential Equations}, 103(2):387--393, 1993.
\newblock \href {https://doi.org/10.1006/jdeq.1993.1056}
  {\path{doi:10.1006/jdeq.1993.1056}}.

\bibitem[Mon02]{montTour}
Richard Montgomery.
\newblock {\em A tour of subriemannian geometries, their geodesics and
  applications}, volume~91 of {\em Mathematical Surveys and Monographs}.
\newblock American Mathematical Society, Providence, RI, 2002.
\newblock \href {https://doi.org/10.1090/surv/091}
  {\path{doi:10.1090/surv/091}}.

\bibitem[Nas56]{nashSmoothIsometric}
John Nash.
\newblock The imbedding problem for {R}iemannian manifolds.
\newblock {\em Annals of Mathematics. Second Series}, 63:20--63, 1956.
\newblock \href {https://doi.org/10.2307/1969989} {\path{doi:10.2307/1969989}}.

\bibitem[Pan16]{pansuCarnotManifold}
Pierre Pansu.
\newblock Submanifolds and differential forms on carnot manifolds, after m.
  gromov and m. rumin.
\newblock 2016.
\newblock \href {http://arxiv.org/abs/1604.06333} {\path{arXiv:1604.06333}}.

\bibitem[Ray68]{raynerFat}
C.~B. Rayner.
\newblock The exponential map for the {L}agrange problem on differentiable
  manifolds.
\newblock {\em Philosophical Transactions of the Royal Society of London.
  Series A. Mathematical and Physical Sciences}, 262:299--344, 1967/68.
\newblock \href {https://doi.org/10.1098/rsta.1967.0052}
  {\path{doi:10.1098/rsta.1967.0052}}.

\bibitem[Tan70]{tanakaDiffSystem}
Noboru Tanaka.
\newblock On differential systems, graded {L}ie algebras and pseudogroups.
\newblock {\em Journal of Mathematics of Kyoto University}, 10:1--82, 1970.
\newblock \href {https://doi.org/10.1215/kjm/1250523814}
  {\path{doi:10.1215/kjm/1250523814}}.

\bibitem[Zan15]{zanetGenericRank4}
Chiara~De Zanet.
\newblock Generic one-step bracket-generating distributions of rank four.
\newblock {\em Archivum Mathematicum}, (5):257--264, 2015.
\newblock \href {https://doi.org/10.5817/am2015-5-257}
  {\path{doi:10.5817/am2015-5-257}}.

\end{thebibliography}

\end{document}